\documentclass[12pt,a4paper]{amsart}
\usepackage{amsmath,amsthm}
\usepackage{enumitem}
\usepackage[noadjust]{cite}
\usepackage[hidelinks]{hyperref}

\usepackage{fullpage}

\usepackage[small]{caption}
\usepackage[subrefformat=parens]{subcaption}
\usepackage{tikz}
\usetikzlibrary{calc}
\usetikzlibrary{arrows}

\newtheorem{theorem}{Theorem}
\newtheorem{lemma}[theorem]{Lemma}
\newtheorem{corollary}[theorem]{Corollary}
\newtheorem{proposition}[theorem]{Proposition}
\newtheorem*{claim}{Claim}
\theoremstyle{definition}

\newtheorem{remark}[theorem]{Remark}
\newtheorem{definition}[theorem]{Definition}

\title{Forbidden patterns of graphs 12-representable by pattern-avoiding words}
\author{Asahi~Takaoka}
\address{
  College of Information and Systems, 
  Muroran Institute of Technology, 
  Mizumoto 27-1, Muroran, 
  Hokkaido, 050--8585, Japan 
}
\email{takaoka@muroran-it.ac.jp}
\date{\today}


\makeatletter
\def\paragraph{\@startsection{paragraph}{4}%
  \z@\z@{-\fontdimen2\font}%
  {\normalfont\bfseries}}
\makeatother

\newcommand{\red}{\textup{red}}

\begin{document}
\maketitle
\begin{abstract}
A graph $G = (\{1, 2, \ldots, n\}, E)$ is \emph{$12$-representable} if 
there is a word $w$ over $\{1, 2, \ldots, n\}$ such that 
two vertices $i$ and $j$ with $i < j$ are adjacent 
if and only if every $j$ occurs before every $i$ in $w$. 
These graphs have been shown to be equivalent to 
the complements of simple-triangle graphs. 
This equivalence provides a characterization in terms of forbidden patterns in vertex orderings 
as well as a polynomial-time recognition algorithm. 
\par
The class of $12$-representable graphs was introduced by Jones \textit{et al.} (2015) 
as a variant of word-representable graphs. 
A general research direction for word-representable graphs 
suggested by Kitaev and Lozin (2015) 
is to study graphs representable 
by some specific types of words. 
For instance, Gao, Kitaev, and Zhang (2017) and Mandelshtam (2019) investigated 
word-representable graphs represented by pattern-avoiding words. 
\par
Following this research direction, this paper studies 
$12$-representable graphs represented by words that avoid a pattern $p$. 
Such graphs are trivial when $p$ is of length $2$. 
When $p = 111$, $121$, $231$, and $321$, 
the classes of such graphs are equivalent to well-known classes, 
such as trivially perfect graphs and bipartite permutation graphs. 
For the cases where $p = 123$, $132$, and $211$, 
this paper provides forbidden pattern characterizations. 

\end{abstract}

\section{Introduction}
\paragraph{Word-representable graphs and $u$-representable graphs.}
Let $w$ be a word over some alphabet. 
Two distinct letters $x$ and $y$ of $w$ are said to be \emph{alternate} if 
either a word $xyxy \cdots$ or a word $yxyx \cdots$ can be obtained 
after removing all other letters from $w$. 
A graph $G$ is \emph{word-representable} 
if there is a word $w$ over the alphabet $V(G)$ 
such that two vertices are adjacent 
if and only if they alternate in $w$. 
Such a word $w$ is called a \emph{word-representant} of $G$. 

Word-representable graphs have been 
well investigated, 
and many results have been obtained~\cite{KL15-book,Kitaev17-LNCS}. 
For example, word-representable graphs can be characterized 
in terms of the orientation of edges, 
called \emph{semi-transitive orientation}~\cite{HKP16-DAM}. 
The recognition problem is NP-complete~\cite{HKP16-DAM}. 
It follows that finding a word-representant is NP-hard. 
The class of word-representable graphs contains 
circle graphs and comparability graphs as proper subclasses~\cite{HKP16-DAM,KP08-JALC}. 

As a generalization of word-representable graphs, 
Jones \textit{et al.}~\cite{JKPR15-EJC} introduced the notion of $u$-representable graphs, 
where $u$ is a word over $\{1, 2\}$ different from $22\cdots2$. 
In this context, word-representable graphs are 
called \emph{$11$-representable}. 
Kitaev~\cite{Kitaev17-JGT} showed that 
if $u$ is of length at least $3$, then 
every labeled graph is $u$-representable. 
This indicates that 
only the following two graph classes are nontrivial 
in the theory of $u$-representable graphs: 
$11$-representable graphs and $12$-representable graphs. 
Note that $21$-representable graphs are exactly $12$-representable graphs~\cite{Kitaev17-JGT}. 
This paper focuses on $12$-representable graphs. 

\paragraph{12-representable graphs.}
Let $[n] = \{1, 2, \ldots, n\}$ for a positive integer $n$. 
A labeled graph $G$ whose labels are drawn from $[n]$ 
is \emph{$12$-representable} if there is a word $w$ over $[n]$ 
containing at least one copy of each letter such that 
two vertices $i$ and $j$ with $i < j$ are adjacent 
if and only if every $j$ occurs before every $i$ in $w$. 
In other words, two vertices $i$ and $j$ are adjacent 
if and only if the word obtained from $w$ after removing 
all letters in $[n] \setminus \{i, j\}$ 
avoids the pattern $12$. 
In this situation, $w$ is said to \emph{$12$-represent} the graph $G$, 
and $w$ is called a \emph{$12$-representant} of $G$. 
For example, the graph in Figure~\ref{fig:co-antenna} 
is $12$-representable by a word $w = 4624153$. 
An unlabeled graph $G$ is $12$-representable if 
there is a labeling of $G$ that generates a $12$-representable labeled graph. 

\begin{figure}[ht]
\centering\begin{tikzpicture}
\def\len{0.8}
\useasboundingbox (-1.8, -0.3) rectangle (1.8, 1.4);
\tikzstyle{every node}=[draw,circle,fill=white,minimum size=5pt,inner sep=0pt]
\node [label=above:4]        (a) at ($(45:\len*1) + (135:\len*1)$) {};
\node [label=above right:1]  (f) at ($(45:\len*1) + (135:\len*0)$) {};
\node [label=below:2]        (e) at ($(45:\len*1) + (135:\len*-1)$) {};
\node [label=below:6]        (c) at (0, 0) {};
\node [label=above left:3]   (d) at ($(45:\len*0) + (135:\len*1)$) {};
\node [label=below:5]         (b) at ($(45:\len*-1) + (135:\len*1)$) {};
\draw [] 
	(a) -- (f) -- (e) -- (c) -- (b) -- (d) -- (a)
	(d) -- (c) -- (f)
;
\end{tikzpicture}
\caption{A graph $12$-representable by a word $w = 4624153$.} 
\label{fig:co-antenna}
\end{figure}
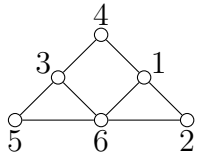

Jones \textit{et al.}~\cite{JKPR15-EJC} showed that 
the class of $12$-representable graphs contains 
co-interval graphs and permutation graphs 
as proper subclasses. 
They also showed that the graph class is 
a proper subclass of comparability graphs, 
which is a subclass of word-representable graphs. 

The $12$-representable graphs are known to have nice characterizations. 
Jones et al.~\cite{JKPR15-EJC} provided 
a characterization of $12$-representable trees 
in terms of forbidden subtrees. 
They also provided 
a necessary condition for $12$-representability, 
which turned out to be sufficient~\cite{Takaoka23-DMGT-inpress}. 
Chen and Kitaev~\cite{CK22-DMGT} investigated 
the $12$-representability of a subclass of grid graphs 
and presented its forbidden induced subgraph characterization. 
Recently, $12$-representable bipartite graphs were 
shown to be equivalent to 
interval containment bigraphs~\cite{Takaoka23-DMGT-inpress}. 
This provides a forbidden induced subgraph characterization 
of $12$-representable bipartite graphs~\cite{FHH99-Combinatorica,STU10-DAM,TM76-DM}. 

The author showed in~\cite{Takaoka23-DMGT-inpress} that 
$12$-representable graphs are exactly 
the complements of simple-triangle graphs~\cite{CK87-CN,Mertzios15-SIAMDM}. 
This equivalence leads to a characterization of $12$-representable graphs 
in terms of forbidden induced \emph{ordered} subgraphs, 
also known as \emph{forbidden patterns}~\cite{Takaoka23-DMGT-inpress}. 
Forbidden induced subgraph characterization 
remains an open problem for $12$-representable graphs~\cite{Takaoka20a-DAM,Takaoka23-DMGT-inpress}. 
The equivalence also indicates that 
$12$-representable graphs can be recognized in polynomial time~\cite{Takaoka23-DMGT-inpress}. 
This implies that 
a $12$-representant can be obtained in polynomial time if it exists. 
Finding an optimal (i.e., shortest) $12$-representant is also investigated. 
Shortest $12$-representants can be obtained in $O(n^2)$ time for labeled graphs, 
whereas for unlabeled graphs, this is an open problem~\cite{Takaoka23a-arXiv}. 
Moreover, the paper~\cite{Takaoka23-DMGT-inpress} suggests studying 
simple-triangle graphs via $12$-representability. 
This paper follows this research direction. 

\paragraph{Graphs representable by some specific types of words.}
As a general research program for word-representable graphs, 
Kitaev and Lozin suggested in~\cite[page~183]{KL15-book} 
to study graphs representable by some specific types of words. 
One example is words containing exactly $k$ copies of each letter, called \emph{$k$-uniform} words. 
Graphs word-representable by $k$-uniform words have been studied 
from the beginning~\cite{KP08-JALC}. 
However, such types of graphs are not interesting for $12$-representable graphs, 
since every $12$-representable graph admits a $12$-representant 
in which each letter occurs at most twice~\cite{JKPR15-EJC}. 

Another example is pattern-avoiding words. 
For instance, Gao, Kitaev, and Zhang~\cite{GKZ17-AJC} studied graphs 
word-representable by $132$-avoiding words. 
They showed that any such graph admits a $132$-avoiding 
representant in which each letter occurs at most twice. 
This indicates that if a graph is word-representable by a $132$-avoiding word, 
then it is necessarily a circle graph. 
Moreover, all trees, cycles, and complete graphs 
are word-representable by $132$-avoiding words. 
On the other hand, 
Mandelshtam~\cite{Mandelshtam19-DMGT} investigated graphs 
word-representable by $123$-avoiding words 
and showed that any such graph admits a $123$-avoiding 
representant in which each letter occurs at most twice; 
hence, it is a circle graph. 
All paths, cycles, and complete graphs but not all trees 
are representable by $123$-avoiding words. 

In this paper, we begin to study graphs 
$12$-representable by pattern-avoiding words. 
Let $p$ be a word (\emph{pattern}) with $\red(p) = p$, 
where $\red(p)$ denotes the word obtained from $p$ by replacing 
each occurrence of the $i$th smallest letter with $i$. 
We say that a graph is \emph{$p$-representable} if 
it is $12$-representable by a $p$-avoiding word. 
This notation is different from that in~\cite{GKZ17-AJC,Mandelshtam19-DMGT}. 
For example, 
the term $123$-representable graph in~\cite{GKZ17-AJC,Mandelshtam19-DMGT} means 
a graph \emph{word-representable} by a $123$-avoiding word, 
while in this paper, it means a graph \emph{$12$-representable} by a $123$-avoiding word.

\paragraph{Forbidden pattern characterization.}
A \emph{vertex ordering} of a graph is a linear ordering of its vertices. 
A graph is \emph{ordered} if it is given with a vertex ordering. 
A \emph{forbidden pattern characterization} of a graph class $\mathcal{G}$ is as follows: 
a graph $G$ is in $\mathcal{G}$ if and only if 
$G$ admits a vertex ordering that contains 
no induced ordered subgraph isomorphic to an ordered graph in 
the set $\mathcal{F_G}$. 
In such a situation, the ordered subgraphs in $\mathcal{F_G}$ are called \emph{forbidden patterns}. 
For example, a graph $G$ is an interval graph if and only if 
there is a vertex ordering of $G$ such that 
for any three vertices $x \prec y \prec z$, 
if $xz \in E(G)$ then $yz \in E(G)$~\cite{Olariu91-IPL}. 
In other words, a graph is an interval graph if and only if 
it admits a vertex ordering that contains no pattern 
in Figure~\ref{fig:fp_simple}\subref{fig:fp_m_int} (Theorem~\ref{thm:fp_m_int}). 

This paper focuses on forbidden pattern characterizations of $p$-representable graphs for several reasons. 
First, $12$-representable graphs and their complements (i.e., simple-triangle graphs) 
can be characterized naturally 
in terms of forbidden patterns~\cite{Takaoka18-DM,Takaoka23-DMGT-inpress}. 
These characterizations lead to the fastest known recognition algorithms, 
which generate a vertex ordering containing no forbidden patterns~\cite{Takaoka20a-DAM,Takaoka23-DMGT-inpress}. 
Given such a vertex ordering, we can obtain a $12$-representant 
in polynomial time~\cite{Takaoka23-DMGT-inpress,Takaoka23a-arXiv}. 

Second, vertex ordering is equivalent to graph labeling, and labeling is important 
when dealing with graphs $12$-representable by pattern-avoiding words. 
We explain this in more detail. 
We say a labeling of a graph \emph{induces} a vertex ordering 
if $x \prec y$ in the ordering whenever the label of $x$ is less than that of $y$. 
(Recall that each vertex has a unique label drawn from $[n]$.) 
Although $12$-representable graphs have been defined in terms of labeled graphs, 
we can use ordered graphs to define them, 
because the linear order induced by the labeling, not the value of the labels, 
is essential for the definition. 
In the following, when we consider a labeled graph, we implicitly assume that 
a vertex ordering induced by the labels is given, and vice versa. 

Labeling is important when dealing with $12$-representable graphs 
because some graphs, such as the path $P_3$ on three vertices, have two different labelings such that 
one is $12$-representable but the other is not~\cite{JKPR15-EJC,Kitaev17-JGT}. 
The same phenomenon can be observed when we consider graphs word-representable by pattern-avoiding words~\cite{GKZ17-AJC,Mandelshtam19-DMGT}. 
Therefore, labeling, and hence vertex ordering, is important 
for graphs $12$-representable by pattern-avoiding words. 

Finally, forbidden pattern characterization is an interesting research topic on its own. 
Many important graph classes, 
including interval graphs, permutation graphs and comparability graphs, 
can be characterized in terms of forbidden patterns. 
The well-known results are surveyed in~\cite{Damaschke90-incollection,FH21-SIDMA} and~\cite[Section 7.4]{BLS99}. 
Other recent examples can be found 
in~\cite{BB23-DAM,CCFHHHS17-DAM,HHMR20-SIDMA,Huang18-DAM,JT19-EJC,SC15-DMTCS,Takaoka18-DM}. 

Recently, Feuilloley and Habib~\cite{FH21-SIDMA} systematically characterized 
all classes of graphs defined by sets of forbidden patterns on three vertices. 
They demonstrated that all such classes can be recognized in polynomial time 
(most of them can be recognized in linear time). 
They also state that an obvious next step is to study graph classes characterized by larger patterns. 
On the other hand, for any $k \geq 4$, there is a forbidden pattern $F$ on $k$ vertices such that recognizing the class of graphs characterized by $F$ is NP-complete~\cite{HMR14-LNCS}. 

Forbidden pattern characterizations sometimes lead to efficient recognition algorithms. 
Indeed, such algorithms were recently proposed for 
adjusted interval digraphs~\cite{FHHR12-DAM,Takaoka21-DAM}, 
simple-triangle graphs~\cite{Takaoka20a-DAM}, and 
interval bigraphs~\cite{Rafiey22-JGT}. 
Moreover, we can exploit forbidden pattern characterizations to design faster algorithms for some combinatorial problems, such as computing maximum matching~\cite{MNN18-SIDMA} and induced matching~\cite{HM20-ALGO}. 
Very recently, a circular version of forbidden pattern characterization was introduced, 
which, for example, can be used to naturally characterize circular-arc graphs and outerplanar graphs~\cite{GHH23-AMC}.

\paragraph{Our results and organization of the paper.}
The classes of $p$-representable graphs are trivial 
when the pattern $p$ is of length $2$ (Proposition~\ref{prop:p=2}). 
In this paper, we investigate $p$-representable graphs when $p$ is of length $3$ 
and show that most of them admit nice characterizations. 

We prove that the classes of $p$-representable graphs 
are equivalent to some well-known classes 
when $p = 111$, $121$, $231$, and $321$. 
For the case where $p = 123$, $132$, and $211$, 
we provide forbidden pattern characterizations. 
These characterizations lead to some results for these classes. 
For instance, the class of 
$123$-representable graphs (resp. $132$-representable graphs and $211$-representable graphs) 
is a subclass of 
max point-tolerance graphs~\cite{CCFHHHS17-DAM,Hixon13-Master,Rusu23-TCS,SC15-DMTCS} 
(resp. co-interval graphs and grounded $L$-graphs~\cite{JT19-EJC}). 
When $p = 112$, characterizing $p$-representable graphs remains an open problem. 
The remaining cases are equivalent to one of the above cases 
(Corollary~\ref{cor:reverse_supplement}). 

The paper is organized as follows. 
In Section~\ref{sec:preliminaries}, we provide some definitions and notations 
as well as some basic results used in this paper. 
In Section~\ref{sec:simple cases}, we deal with $111$, $121$, $231$, and $321$-representable graphs. 
In Sections~\ref{sec:123},~\ref{sec:132}, and~\ref{sec:211}, we provide 
the forbidden pattern characterizations 
for $123$, $132$, and $211$-representable graphs, respectively. 
In Section~\ref{sec:conclusion}, 
we conclude this paper and discuss further research directions.

\section{Preliminaries}\label{sec:preliminaries}
In this section, we introduce some definitions and notations 
as well as some basic results used in this paper.

\subsection{Words and graphs}
Let $[n] = \{1, 2, \ldots, n\}$ for a positive integer $n$. 
For a word $w$ over $[n]$ and a subset $S \subseteq [n]$, 
let $w_S$ denote a word obtained from $w$ 
after removing all letters in $[n] \setminus S$. 
The \emph{reduced form} of a word $w$, 
denoted by $\red(w)$, is the word obtained from $w$ by replacing 
each occurrence of the $i$th smallest letter with $i$. 
Let $p = p_1 p_2 \cdots p_k$ be a word (\emph{pattern}) with $\red(p) = p$. 
A word $w = w_1 w_2 \cdots w_{\ell}$ \emph{contains} the pattern $p$ 
if there are indices $1 \leq i_1 < i_2 < \cdots < i_k \leq \ell$ 
such that $\red(w_{i_1} w_{i_2} \ldots w_{i_k}) = p$. 
A word $w$ \emph{avoids} the pattern $p$ if it does not contain $p$. 
In this case, the word $w$ is called \emph{$p$-avoiding}. 

We assume that all graphs in this paper are finite, simple, and undirected. 
We write $uv$ for the edge joining two vertices $u$ and $v$. 
For a graph $G$, we write $V(G)$ and $E(G)$ 
for the vertex set and the edge set of $G$, respectively. 
We usually denote the number of vertices by $n$. 
The \emph{complement} of a graph $G$ is the graph $\overline{G}$ 
such that $V(\overline{G}) = V(G)$ and 
$uv \in E(\overline{G}) \iff uv \notin E(G)$ 
for any two vertices $u, v \in V(\overline{G})$. 
For a graph $G$, a graph $H$ with $V(H) \subseteq V(G)$ 
is an \emph{induced subgraph} of $G$ 
if for all $u, v \in V(H)$, 
$uv \in E(H) \iff uv \in E(G)$. 

A \emph{labeled graph} of a graph $G$ is obtained from $G$ 
by assigning an integer (\emph{label}) to each vertex. 
A \emph{labeling} of $G$ is an assignment of labels to the vertices of $G$. 
All labels are assumed to be distinct and drawn from $[n]$. 
For a labeled graph, we usually denote its vertices by their labels. 
Unless stated otherwise, graphs are assumed to be unlabeled.

\subsection{Graph classes}
We now describe the fundamental graph classes that appear in this paper. 
For a comprehensive survey on graph classes, see~\cite{BLS99,Golumbic04,Spinrad03}. 

A graph is an \emph{intersection graph} 
if there is a set of objects such that 
each vertex corresponds to an object 
and two vertices are adjacent if and only if 
the two corresponding objects have a nonempty intersection. 
The set of objects is called a \emph{representation} or \emph{model} of the graph. 
We use the term model in this paper to avoid confusion. 

An \emph{interval graph} is the intersection graph 
of intervals on the real line. 
The complement of an interval graph is called a \emph{co-interval graph}. 
An interval graph is \emph{unit} if it admits a model 
in which every interval has unit length, and 
\emph{proper} if no interval properly contains another. 
An interval graph is unit 
if and only if it is proper~\cite{BW99-DM}. 
An interval graph is a \emph{trivially perfect graph} if it admits a model 
in which for any two intervals, either they are disjoint or one contains the other. 

Let $\pi$ be a permutation over $[n]$. 
The \emph{permutation graph} of $\pi$ is a graph 
with vertex set $[n]$ such that 
two vertices $i$ and $j$ with $i < j$ are adjacent 
if and only if $j$ occurs before $i$ in $\pi$. 
A \emph{bipartite permutation graph} is a permutation graph that is bipartite. 

We need the forbidden pattern characterizations of the graphs above. 
See Figure~\ref{fig:fp_simple} for the forbidden patterns. 
To draw patterns, as in~\cite{FH21-SIDMA}, we use 
lines and dashed lines to denote edges and non-edges, respectively, 
while edges that may or may not be present are not drawn. 

\begin{theorem}[\cite{BLS99,Damaschke90-incollection,FH21-SIDMA}]\label{thm:fp_m_int}
A graph is an interval graph if and only if 
it has a vertex ordering 
that does not contain any pattern in Figure~\ref{fig:fp_simple}\subref{fig:fp_m_int}. 
\end{theorem}

\begin{theorem}[\cite{BLS99,Damaschke90-incollection,FH21-SIDMA}]\label{thm:fp_pg}
A graph is a permutation graph if and only if 
it has a vertex ordering 
that does not contain any pattern 
in Figures~\ref{fig:fp_simple}\subref{fig:fp_comp} and~\ref{fig:fp_simple}\subref{fig:fp_cocomp}. 
\end{theorem}

\begin{theorem}[\cite{BLS99,Damaschke90-incollection,FH21-SIDMA}]\label{thm:fp_tpg}
A graph is a trivially perfect graph if and only if 
it has a vertex ordering 
that does not contain any pattern 
in Figures~\ref{fig:fp_simple}\subref{fig:fp_m_int} and~\ref{fig:fp_simple}\subref{fig:fp_comp}. 
\end{theorem}

\begin{theorem}[\cite{BLS99,Damaschke90-incollection,FH21-SIDMA}]\label{thm:fp_bpg}
A graph is a bipartite permutation graph if and only if 
it has a vertex ordering 
that does not contain any pattern 
in Figures~\ref{fig:fp_simple}\subref{fig:fp_comp},~\ref{fig:fp_simple}\subref{fig:fp_cocomp}, and~\ref{fig:fp_simple}\subref{fig:fp_triangle}. 
\end{theorem}

\begin{figure}[ht]
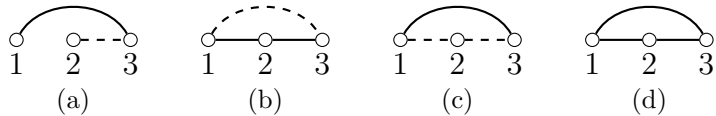

  \centering
  \subcaptionbox{\label{fig:fp_m_int}}{\begin{tikzpicture}
\input{./fig/pattern/template_fp3}
\draw [thick,dashed] (y) -- (z);
\draw [thick] (x) to [out=60, in=120] (z);
\end{tikzpicture}}
  \subcaptionbox{\label{fig:fp_comp}}{\begin{tikzpicture}
\input{./fig/pattern/template_fp3}
\draw [thick] (x) -- (y);
\draw [thick] (y) -- (z);
\draw [thick,dashed] (x) to [out=60, in=120] (z);
\end{tikzpicture}}
  \subcaptionbox{\label{fig:fp_cocomp}}{\begin{tikzpicture}
\input{./fig/pattern/template_fp3}
\draw [thick,dashed] (x) -- (y);
\draw [thick,dashed] (y) -- (z);
\draw [thick] (x) to [out=60, in=120] (z);
\end{tikzpicture}}
  \subcaptionbox{\label{fig:fp_triangle}}{\begin{tikzpicture}
\input{./fig/pattern/template_fp3}
\draw [thick] (x) -- (y);
\draw [thick] (y) -- (z);
\draw [thick] (x) to [out=60, in=120] (z);
\end{tikzpicture}}
  \caption{Forbidden patterns. 
    Lines and dashed lines denote edges and non-edges, respectively. 
    Edges that may or may not be present are not drawn. 
  }
  \label{fig:fp_simple}
\end{figure}

Let $\mathcal{L}_1$ and $\mathcal{L}_2$ be two horizontal lines 
in the $xy$-plane with $\mathcal{L}_1$ above $\mathcal{L}_2$. 
A point on $\mathcal{L}_1$ and an interval on $\mathcal{L}_2$ 
define a triangle between two lines. 
A \emph{simple-triangle graph}~\cite{CK87-CN,Mertzios15-SIAMDM} is 
the intersection graph of such triangles; 
see Figure~\ref{fig:antenna} for example. 
Triangles can be degenerated to lines, as in the example. 
Simple-triangle graphs are exactly the complements of 
$12$-representable graphs~\cite{Takaoka23-DMGT-inpress}. 

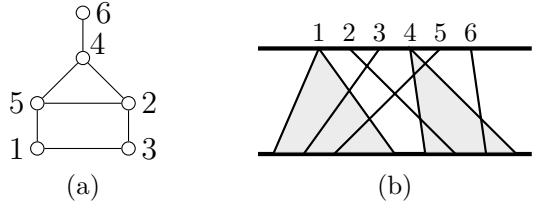
\begin{figure}[ht]
  \centering
  \subcaptionbox{\label{fig:antenna_graph}}{\begin{tikzpicture}
\def\len{0.6}
\useasboundingbox (-2, -0.1) rectangle (2, 2.0);
\tikzstyle{every node}=[draw,circle,fill=white,minimum size=5pt,inner sep=0pt]
\node [label=right:6]  (c) at (0, 3*\len) {};
\node [label=above right:4]  (a) at (0, 2*\len) {};
\node [label=left:5]   (b) at (-1*\len, \len) {};
\node [label=right:3]  (d) at ( 1*\len, 0) {};
\node [label=right:2]  (e) at ( 1*\len, \len) {};
\node [label=left:1]   (f) at (-1*\len, 0) {};
\draw [] 
	(c) -- (a) -- (e) -- (d) -- (f) -- (b) -- (a)
	(e) -- (b)
;
\end{tikzpicture}}
  \subcaptionbox{\label{fig:antenna_model}}{\begin{tikzpicture}
\def\len{0.2}
\tikzstyle{every node}=[minimum size=5pt, inner sep=0pt]
\def\La{7.5*\len}
\def\Lb{0.5*\len}
\draw [ultra thick] (0, \La) -- (18*\len, \La);
\draw [ultra thick] (0, \Lb) -- (18*\len, \Lb);
\def\pa{(4*\len, \La)}
\def\pb{(6*\len, \La)}
\def\pc{(8*\len, \La)}
\def\pd{(10*\len, \La)}
\def\pe{(12*\len, \La)}
\def\pf{(14*\len, \La)}
\node [label=above:{\footnotesize 1}] at \pa {};
\node [label=above:{\footnotesize 2}] at \pb {};
\node [label=above:{\footnotesize 3}] at \pc {};
\node [label=above:{\footnotesize 4}] at \pd {};
\node [label=above:{\footnotesize 5}] at \pe {};
\node [label=above:{\footnotesize 6}] at \pf {};
\def\la{(1*\len, \Lb)}
\def\lb{(13*\len, \Lb)}
\def\lc{(3*\len, \Lb)}
\def\ld{(11*\len, \Lb)}
\def\le{(5*\len, \Lb)}
\def\lf{(15*\len, \Lb)}
\def\ra{(9*\len, \Lb)}
\def\rb{(13*\len, \Lb)}
\def\rc{(3*\len, \Lb)}
\def\rd{(17*\len, \Lb)}
\def\re{(5*\len, \Lb)}
\def\rf{(15*\len, \Lb)}
\def\op{0.5}
\draw [fill=gray!30, opacity=\op] 
	\pa -- \la -- \ra --  cycle
	\pd -- \ld -- \rd --  cycle
	\pb -- \lb -- \rb --  cycle
	\pc -- \lc -- \rc --  cycle
	\pe -- \le -- \re --  cycle
	\pf -- \lf -- \rf --  cycle
;
\draw [thick]
	\pa -- \la -- \ra --  cycle
	\pd -- \ld -- \rd --  cycle
	\pb -- \lb -- \rb --  cycle
	\pc -- \lc -- \rc --  cycle
	\pe -- \le -- \re --  cycle
	\pf -- \lf -- \rf --  cycle
;
\end{tikzpicture}}
  \caption{
    (a) The complement of the graph in Figure~\ref{fig:co-antenna}, and 
    (b) its model. 
  }
  \label{fig:antenna}
\end{figure}

\subsection{12-representable graphs}
The following is a basic characterization of $12$-representable graphs. 
Note that the necessity 
was first presented by Jones \textit{et al.}~\cite{JKPR15-EJC}. 
\begin{theorem}[\cite{Takaoka23-DMGT-inpress}]\label{thm:fp}
A graph is $12$-representable if and only if 
it has a vertex ordering 
that does not contain any pattern in Figure~\ref{fig:fp}. 
\end{theorem}

\begin{figure}[ht]
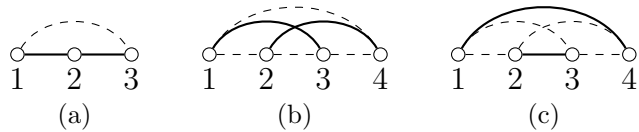

  \centering
  \subcaptionbox{\label{fig:fp1}}{\begin{tikzpicture}
\input{./fig/pattern/template_fp3}
\draw [thick] (x) -- (y);
\draw [thick] (y) -- (z);
\draw [dashed] (x) to [out=60, in=120] (z);
\end{tikzpicture}}
  \subcaptionbox{\label{fig:fp2}}{\begin{tikzpicture}
\input{./fig/pattern/template_fp4}
\draw [thick] (x) to [out=60, in=120] (z);
\draw [thick] (y) to [out=60, in=120] (w);
\draw [dashed] (x) to [out=60, in=120] (w);
\draw [dashed] (x) -- (y);
\draw [dashed] (y) -- (z);
\draw [dashed] (z) -- (w);
\end{tikzpicture}}
  \subcaptionbox{\label{fig:fp3}}{\begin{tikzpicture}
\input{./fig/pattern/template_fp4}
\draw [thick] (x) to [out=60, in=120] (w);
\draw [thick] (y) to (z);
\draw [dashed] (x) -- (y);
\draw [dashed] (x) to [out=60, in=120] (z);
\draw [dashed] (y) to [out=60, in=120] (w);
\draw [dashed] (z) -- (w);
\end{tikzpicture}}
  \caption{Forbidden patterns of $12$-representable graphs. }
  \label{fig:fp}
\end{figure}

We also need the following results. 
\begin{theorem}[\cite{JKPR15-EJC}]\label{thm:at most twice}
If a labeled graph is $12$-representable, then 
there is a $12$-representant in which each letter occurs at most twice. 
\end{theorem}

\begin{theorem}[\cite{JKPR15-EJC}]\label{thm:permutation}
A graph is $12$-representable by a permutation if and only if it is a permutation graph. 
\end{theorem}

\begin{theorem}[\cite{JKPR15-EJC}]\label{thm:cycle}
Any cycle of length at least 5 is not $12$-representable. 
\end{theorem}

\subsection{$p$-representable graphs}
We prove the basic properties of $p$-representable graphs. 

\begin{proposition}\label{prop:p=2}
A $p$-representable graph is complete if $p = 12$ and edgeless if $p = 21$. 
A $p$-representable graph is a permutation graph if $p = 11$. 
\end{proposition}
\begin{proof}
A $12$-avoiding (resp. $21$-avoiding) word over $[n]$ is 
a descending (resp. ascending) sequence of the letters. 
Thus, every (resp. no) pair of vertices are adjacent 
in a graph $12$-represented by a $12$-avoiding (resp. $21$-avoiding) word. 
A $11$-avoiding word over $[n]$ is a permutation of $[n]$. 
Thus, the second statement follows from Theorem~\ref{thm:permutation}. 
\end{proof}

A claim similar to Theorem~\ref{thm:at most twice} holds for $p$-representable graphs. 
\begin{proposition}\label{prop:labeled induced subgraph}
If a labeled graph is $p$-representable, 
then there is a $p$-avoiding $12$-representant 
in which each letter occurs at most twice. 
\end{proposition}
\begin{proof}
The proof is the same as that of Theorem~\ref{thm:at most twice}. 
Let $G$ be a graph and $w$ be a $p$-avoiding representant of $G$. 
Suppose that a letter $i$ occurs more than twice. 
Let $w = w_1iw_2iw_3$, where $w_1$, $w_2$, and $w_3$ are subwords of $w$ such that 
$w_1$ and $w_3$ contain no copies of $i$. 
Each copy of $i$ in $w_2$ can be omitted; that is, 
the word $w'$ obtained from $w$ by removing all copies of $i$ in $w_2$ 
is a $12$-representant of $G$ and still $p$-avoiding. 
Performing the same procedure with all other letters 
occurring more than twice in $w'$ yields a required word. 
\end{proof}
Thus, in the following, we assume that every $12$-representant contains at most two letters. 

The classes of $p$-representable graphs are hereditary. 
\begin{proposition}\label{prop:labeled induced subgraph}
If a labeled graph $G$ is $p$-representable, 
then every induced subgraph of $G$ is $p$-representable. 
\end{proposition}
\begin{proof}
Let $H$ be an induced subgraph of $G$. 
If $w$ is a $p$-avoiding representant of $G$, 
then $w_{V(H)}$ is a $12$-representant of $H$ and still $p$-avoiding. 
\end{proof}

There exist two different patterns $p$ and $p'$ such that 
a graph is $p$-representable if and only if $p'$-representable. 
The following is an example. 
The \emph{reverse} of a word $w = w_1 w_2 \cdots w_k$ over $[n]$ is the word $r(w) = w_k w_{k-1} \cdots w_1$. 
The \emph{complement} of $w$ is the word $c(w) = (n + 1 - w_1) (n + 1 - w_2) \cdots (n + 1 - w_k)$. 
\begin{proposition}\label{prop:reverse_supplement}
A graph is $p$-representable if and only if it is $c(r(p))$-representable. 
\end{proposition}
\begin{proof}
Let $G$ be a labeled graph with $n$ vertices. 
The graph $G$ is said to be \emph{$21$-representable} by a word $w$ if 
two vertices $i$ and $j$ with $i < j$ are adjacent 
if and only if every $i$ occurs before every $j$ in $w$. 
The \emph{supplement} of $G$, denoted by $c(G)$, is a labeled graph 
obtained by relabeling $G$ so that each label $i$ is replaced by $n+1-i$. 
Now, a word $w$ is a $p$-avoiding $12$-representant of $G$ if and only if 
$r(w)$ is an $r(p)$-avoiding $21$-representant of $G$ if and only if 
$c(r(w))$ is a $c(r(p))$-avoiding $12$-representant of $c(G)$. 
Since $G$ and $c(G)$ are the same graph with different labels, the claim holds. 
\end{proof}

Proposition~\ref{prop:reverse_supplement} yields the following corollary, 
which means there are at most eight nonequivalent classes 
of $p$-representable graphs if $p$ is of length $3$. 
\begin{corollary}\label{cor:reverse_supplement}
Two classes of $p$-representable graphs are equivalent in the following cases: 
$p = 112$ and $p = 122$; 
$p = 121$ and $p = 212$; 
$p = 211$ and $p = 221$; 
$p = 132$ and $p = 213$; 
$p = 231$ and $p = 312$. 
\end{corollary}

\section{$p$-representable graphs equivalent to known classes}\label{sec:simple cases}
In this section, we show that 
$111$, $121$, $231$, and $321$-representable graphs are precisely 
$12$-representable graphs, 
permutation graphs, 
trivially perfect graphs, and 
bipartite permutation graphs, respectively.

\subsection{111-representable graphs}
\begin{theorem}
A graph is $12$-representable by a $111$-avoiding word 
if and only if it is $12$-representable. 
\end{theorem}
\begin{proof}
The necessity is obvious. 
Theorem~\ref{thm:at most twice} indicates the sufficiency. 
\end{proof}

\subsection{121-representable graphs}
\begin{lemma}\label{lem:121}
A labeled graph $G$ is $12$-representable by a $121$-avoiding word 
if and only if there is a permutation $12$-representing $G$. 
\end{lemma}
\begin{proof}
The sufficiency is obvious, and we will prove the necessity. 
Let $w$ be a $121$-avoiding representant of $G$. 
Suppose that a letter $i$ occurs twice in $w$. 
Let $w = w_1iw_2iw_3$, where $w_1$, $w_2$, and $w_3$ are subwords of $w$. 
Since $w$ is $121$-avoiding, every letter in $w_2$ is less than $i$ or $w_2$ is empty. 
Thus, the first copy of $i$ can be omitted; that is, 
$w' = w_1w_2iw_3$ is a $12$-representant of $G$ and still $121$-avoiding. 
Performing the same procedure with all other letters 
occurring twice in $w'$ yields a permutation representing $G$. 
\end{proof}

Lemma~\ref{lem:121} together with Theorem~\ref{thm:permutation} indicates the following. 
\begin{theorem}\label{thm:121}
A graph is $12$-representable by a $121$-avoiding word if and only 
if it is a permutation graph. 
\end{theorem}

The forbidden patterns of permutation graphs (Theorem~\ref{thm:fp_pg}) also characterize $121$-representable graphs. 
To show this, we prove the following. 
\begin{theorem}\label{thm:permutation_labeled}
A labeled graph is $12$-representable by a permutation if and only if 
it does not contain any pattern 
in Figures~\ref{fig:fp_simple}\subref{fig:fp_comp} and~\ref{fig:fp_simple}\subref{fig:fp_cocomp} 
\end{theorem}
\begin{proof}
To prove the necessity, 
by Proposition~\ref{prop:labeled induced subgraph}, 
it suffices to show that any pattern 
in Figures~\ref{fig:fp_simple}\subref{fig:fp_comp} and~\ref{fig:fp_simple}\subref{fig:fp_cocomp} 
is not $12$-representable by a permutation. 
Theorem~\ref{thm:fp} implies that 
the pattern in Figure~\ref{fig:fp_simple}\subref{fig:fp_comp} is not $12$-representable. 
Let $H$ be a graph with $V(H) = \{1, 2, 3\}$ such that $13 \in E(H)$ and $12, 23 \notin E(H)$. 
Suppose that there is a $12$-representant $w$ of $H$. 
Since $13 \in E(H)$, 
the rightmost $3$ occurs before the leftmost $1$. 
Since $12 \notin E(H)$, some $2$ occurs after the leftmost $1$. 
Since $23 \notin E(H)$, some $2$ occurs before the rightmost $3$. 
Thus, the pattern $2312$ must occur in $w$; that is, $w$ is not a permutation. 
Hence, the pattern in Figure~\ref{fig:fp_simple}\subref{fig:fp_cocomp} 
is not $12$-representable by a permutation. 
\par
To prove the sufficiency, 
we consider the orientation of the complete graph on $[n]$ 
defined so that for any two vertices $i$ and $j$ with $i < j$, 
we have $i \to j$ if $ij \notin E(G)$ and $j \to i$ otherwise. 
The orientation is acyclic (see, e.g.,~\cite[Chapter 7]{Golumbic04}), 
and hence, it determines a unique permutation $\pi$ on $[n]$. 
By definition, $\pi$ is a $12$-representant of $G$. 
\end{proof}

\begin{theorem}\label{cor:121_labeled}
A labeled graph is $12$-representable by a $121$-avoiding word if and only if 
it does not contain any pattern 
in Figures~\ref{fig:fp_simple}\subref{fig:fp_comp} and~\ref{fig:fp_simple}\subref{fig:fp_cocomp}. 
\end{theorem}
\begin{proof}
The sufficiency is obvious from Theorem~\ref{thm:permutation_labeled}, and we will prove the necessity. 
Theorem~\ref{thm:fp} implies that 
the pattern in Figure~\ref{fig:fp_simple}\subref{fig:fp_comp} is not $12$-representable. 
Let $H$ be a graph with $V(H) = \{1, 2, 3\}$ such that $13 \in E(H)$ and $12, 23 \notin E(H)$. 
Suppose that there is a $12$-representant $w$ of $H$. 
As shown in the proof of Theorem~\ref{thm:permutation_labeled}, 
the pattern $2312$ must occur in $w$. 
Thus, $w$ contains the pattern $121$. 
Therefore, the pattern in Figure~\ref{fig:fp_simple}\subref{fig:fp_cocomp} 
is not $12$-representable by a $121$-avoiding word. 
\end{proof}

\subsection{231-representable graphs}
\begin{lemma}\label{lem:231}
A labeled graph $G$ is $12$-representable by a $231$-avoiding word 
if and only if there is a $231$-avoiding permutation $12$-representing $G$. 
\end{lemma}
\begin{proof}
The sufficiency is obvious, and we will prove the necessity. 
Let $w$ be a $231$-avoiding representant of $G$. 
Suppose that a letter $i$ occurs twice in $w$. 
Let $w = w_1iw_2iw_3$, where $w_1$, $w_2$, and $w_3$ are subwords of $w$. 
If $w_2$ contains two letters $j$ and $k$ with $j > i > k$ 
such that $j$ occurs before $k$ in $w_2$, 
then the first copy of $i$ together with $j$ and $k$ form the pattern $231$, 
a contradiction. 
Thus, $w_2$ can be partitioned into two parts $w_2 = w_2'w_2''$ such that 
every letter of $w_2'$ is less than $i$ and 
every letter of $w_2''$ is larger than $i$. 
Therefore, the word $w' = w_1w_2'iw_2''w_3$ is a $12$-representant of $G$. 
If $w$ is $231$-avoiding, then $w'$ is also $231$-avoiding. 
Performing the same procedure with all other letters 
occurring twice in $w'$ yields a $231$-avoiding permutation representing $G$. 
\end{proof}

A $231$-avoiding permutation is called a \emph{stack-sortable permutation}, 
and the permutation graph of a stack-sortable permutation is 
exactly a trivially perfect graph~\cite{Rotem81-DM}. 
Thus, the following is a direct consequence of Lemma~\ref{lem:231} and Theorem~\ref{thm:permutation}. 
\begin{theorem}\label{thm:231}
A graph is $12$-representable by a $231$-avoiding word if and only if 
it is a trivially perfect graph. 
\end{theorem}

The forbidden patterns of trivially perfect graphs (Theorem~\ref{thm:fp_tpg}) also characterize $231$-representable graphs. 
\begin{theorem}\label{cor:231_labeled}
A labeled graph is $12$-representable by a $231$-avoiding word if and only if 
it does not contain any pattern 
in Figures~\ref{fig:fp_simple}\subref{fig:fp_m_int} and~\ref{fig:fp_simple}\subref{fig:fp_comp}. 
\end{theorem}
\begin{proof}
To prove the necessity, 
by Proposition~\ref{prop:labeled induced subgraph}, 
it suffices to show that any pattern 
in Figures~\ref{fig:fp_simple}\subref{fig:fp_m_int} and~\ref{fig:fp_simple}\subref{fig:fp_comp} 
is not $12$-representable by a $231$-avoiding word. 
Theorem~\ref{thm:fp} implies that 
the pattern in Figure~\ref{fig:fp_simple}\subref{fig:fp_comp} is not $12$-representable. 
Let $H$ be a graph with $V(H) = \{1, 2, 3\}$ such that $13 \in E(H)$ and $23 \notin E(H)$. 
Suppose that there is a $12$-representant $w$ of $H$. 
Since $13 \in E(H)$, 
the rightmost $3$ occurs before the leftmost $1$. 
Since $23 \notin E(H)$, some $2$ occurs before the rightmost $3$. 
Thus, the pattern $231$ occurs in $w$. 
Hence, the pattern in Figure~\ref{fig:fp_simple}\subref{fig:fp_m_int} 
is not $12$-representable by a $231$-avoiding word. 
\par
To prove the sufficiency, 
we consider a labeled graph $G$ containing no pattern 
in Figures~\ref{fig:fp_simple}\subref{fig:fp_m_int} and~\ref{fig:fp_simple}\subref{fig:fp_comp}. 
Theorem~\ref{thm:permutation_labeled} indicates that there is a permutation $\pi$ on $[n]$ 
representing $G$. 
If $\pi$ contains a pattern $231$, then $G$ contains the pattern 
$H$ with $V(H) = \{1, 2, 3\}$ such that $12, 13 \in E(H)$ and $23 \notin E(H)$, 
a contradiction. 
Thus, $\pi$ is $231$-avoiding. 
\end{proof}

\subsection{321-representable graphs}
\begin{lemma}\label{lem:321:permutation}
A labeled graph $G$ is $12$-representable by a $321$-avoiding word 
if and only if there is a $321$-avoiding permutation $12$-representing $G$. 
\end{lemma}
\begin{proof}
The sufficiency is obvious, and we will prove the necessity. 
Let $w$ be a 321-avoiding representant of $G$. 
Suppose that a letter $i$ occurs twice in $w$. 
Let $w = w_1iw_2iw_3$, where $w_1$, $w_2$, and $w_3$ are subwords of $w$. 
If $w_2$ is empty or every letter in $w_2$ is less than $i$, 
then the first copy of $i$ can be omitted; that is, 
$w' = w_1w_2iw_3$ is a $12$-representant of $G$ and still $321$-avoiding. 
Similarly, if every letter in $w_2$ is larger than $i$, 
then the second copy of $i$ can be omitted. 
Now, suppose that $w_2$ contains 
two letters $j$ and $k$ with $j > i > k$. 
No letters in $w_1$ are larger than $i$; 
otherwise, the pattern $321$ would exist. 
Recall that $w$ contains at least one copy of each letter in $[n]$. 
Thus, all letters larger than $i$ occur after the first copy of $i$. 
Similarly, no letters in $w_3$ are less than $i$, and hence, 
all letters less than $i$ occur before the second copy of $i$. 
Therefore, the vertex labeled with $i$ is isolated in $G$. 
We relabel $G$ so that label $i$ is replaced by $n+1$. 
The word $w'' = w_1w_2w_3(n+1)$ is a $12$-representant of the relabeled graph and still $321$-avoiding. 
Performing the same procedure with all letters 
occurring twice in $w$ yields a $321$-avoiding permutation representing $G$. 
\end{proof}

The following lemma together with Theorem~\ref{thm:cycle} indicates that 
a graph is bipartite if it is $12$-representable by a $321$-avoiding word. 
\begin{lemma}\label{lem:321:cycle}
Any cycle of length $3$ is not $12$-representable by a $321$-avoiding word. 
\end{lemma}
\begin{proof}
Labeling of the cycle is unique because it is complete. 
Since each vertex is adjacent to the others, 
every $12$-representant contains the pattern $321$. 
\end{proof}

The lemmas above yield the following. 
\begin{theorem}\label{thm:321}
A graph is $12$-representable by a $321$-avoiding word if and only 
if it is a bipartite permutation graph. 
\end{theorem}
\begin{proof}
Let $G$ be a $321$-representable graph. 
Lemma~\ref{lem:321:permutation} and Theorem~\ref{thm:permutation} indicate 
that $G$ is a permutation graph. 
Lemma~\ref{lem:321:cycle} and Theorem~\ref{thm:cycle} indicate 
that $G$ is bipartite. 
Conversely, let $G$ be the permutation graph of a permutation $\pi$ on $[n]$ that is bipartite. 
Two vertices $i$ and $j$ with $i < j$ are adjacent in $G$ 
if and only if $j$ occurs before $i$ in $\pi$. 
Thus, $\pi$ is a $12$-representant of $G$. 
If $\pi$ contains the pattern $321$, 
then $G$ contains a cycle of length $3$, a contradiction. 
Thus, $\pi$ is $321$-avoiding. 
\end{proof}

\begin{remark}
The pattern in Figure~\ref{fig:fp_simple}\subref{fig:fp_cocomp} 
is $12$-representable by a $321$-avoiding word $w = 2312$. 
Thus, the class of $321$-representable graphs cannot be characterized by 
the forbidden patterns of bipartite permutation graphs (Theorem~\ref{thm:fp_bpg}). 
\end{remark}

\section{$123$-representable graphs}\label{sec:123}
In this section, we show a forbidden pattern characterization 
and the related results for $123$-representable graphs. 
\begin{lemma}\label{lem:123-1}
If a labeled graph contains a pattern in Figure~\ref{fig:fp123}, 
then it is not $12$-representable by a $123$-avoiding word. 
\end{lemma}
\begin{proof}
By Proposition~\ref{prop:labeled induced subgraph}, 
it suffices to show that any pattern in Figure~\ref{fig:fp123} 
is not $12$-representable by a $123$-avoiding word. 
Theorem~\ref{thm:fp} implies that 
the pattern in Figure~\ref{fig:fp123}\subref{fig:fp123-1} is not $12$-representable. 
\par
Let $H_1$ be a graph with 
$V(H_1) = \{1, 2, 3, 4\}$ such that $23 \in E(H_1)$ and $13, 24 \notin E(H_1)$. 
Suppose that there is a $12$-representant $w_1$ of $H_1$. 
Since $23 \in E(H_1)$, 
the rightmost $3$ occurs before the leftmost $2$ in $w_1$. 
Since $13 \notin E(H_1)$, some $1$ occurs before the rightmost $3$. 
Since $24 \notin E(H_1)$, some $4$ occurs after the leftmost $2$. 
Thus, the pattern $1324$ must occur in $w_1$. 
Hence, $w_1$ contains the pattern $123$. 
Therefore, the pattern in Figure~\ref{fig:fp123}\subref{fig:fp123-2} 
is not $12$-representable by a $123$-avoiding word. 
\par
Let $H_2$ be a graph with 
$V(H_2) = \{1, 2, 3, 4\}$ such that $13 \in E(H_2)$ and $14, 23 \notin E(H_2)$. 
Suppose that there is a $12$-representant $w_2$ of $H_2$. 
Since $13 \in E(H_2)$, 
the rightmost $3$ occurs before the leftmost $1$ in $w_2$. 
Since $14 \notin E(H_2)$, some $4$ occurs after the leftmost $1$. 
Since $23 \notin E(H_2)$, some $2$ occurs before the rightmost $3$. 
Thus, the pattern $2314$ must occur in $w_2$. 
Hence, $w_2$ contains the pattern $123$. 
Therefore, the pattern in Figure~\ref{fig:fp123}\subref{fig:fp123-3} 
is not $12$-representable by a $123$-avoiding word. 
A similar argument would show the case for 
the pattern in Figure~\ref{fig:fp123}\subref{fig:fp123-4}. 
\end{proof}

\begin{figure}[ht]
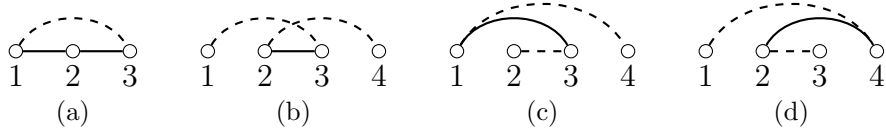

  \centering
  \subcaptionbox{\label{fig:fp123-1}}{\begin{tikzpicture}
\input{./fig/pattern/template_fp3}
\draw [thick] (x) -- (y);
\draw [thick] (y) -- (z);
\draw [thick,dashed] (x) to [out=60, in=120] (z);
\end{tikzpicture}}
  \subcaptionbox{\label{fig:fp123-2}}{\begin{tikzpicture}
\input{./fig/pattern/template_fp4}
\draw [thick,dashed] (x) to [out=60, in=120] (z);
\draw [thick,dashed] (y) to [out=60, in=120] (w);
\draw [thick] (y) to (z);
\end{tikzpicture}}
  \subcaptionbox{\label{fig:fp123-3}}{\begin{tikzpicture}
\input{./fig/pattern/template_fp4}
\draw [thick,dashed] (x) to [out=60, in=120] (w);
\draw [thick] (x) to [out=60, in=120] (z);
\draw [thick,dashed] (y) to (z);
\end{tikzpicture}}
  \subcaptionbox{\label{fig:fp123-4}}{\begin{tikzpicture}
\input{./fig/pattern/template_fp4}
\draw [thick,dashed] (x) to [out=60, in=120] (w);
\draw [thick] (y) to [out=60, in=120] (w);
\draw [thick,dashed] (y) to (z);
\end{tikzpicture}}
  \caption{Forbidden patterns of $123$-representable graphs. }
  \label{fig:fp123}
\end{figure}

We show that the forbidden patterns in Figure~\ref{fig:fp123} 
characterize $123$-representable graphs. 
To this end, we consider the complements of the patterns in Figure~\ref{fig:fp123}; 
see Figure~\ref{fig:cfp123}. 
We first deal with graphs characterized 
by the forbidden pattern in Figure~\ref{fig:cfp123}\subref{fig:cfp123-2}. 
Such graphs are called 
\emph{max point-tolerance (MPT) graphs}~\cite{CCFHHHS17-DAM,Rusu23-TCS}, 
which are also known as 
\emph{hook graphs}~\cite{Hixon13-Master} and 
\emph{$p$-Box(1) graphs}~\cite{SC15-DMTCS}. 

\begin{figure}[ht]
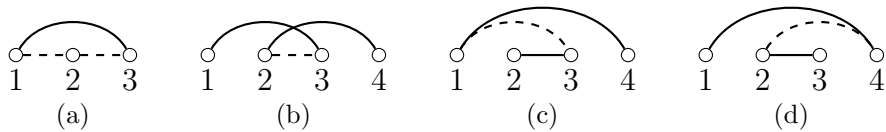

  \centering
  \subcaptionbox{\label{fig:cfp123-1}}{\begin{tikzpicture}
\input{./fig/pattern/template_fp3}
\draw [thick,dashed] (x) -- (y);
\draw [thick,dashed] (y) -- (z);
\draw [thick] (x) to [out=60, in=120] (z);
\end{tikzpicture}}
  \subcaptionbox{\label{fig:cfp123-2}}{\begin{tikzpicture}
\input{./fig/pattern/template_fp4}
\draw [thick] (x) to [out=60, in=120] (z);
\draw [thick] (y) to [out=60, in=120] (w);
\draw [thick,dashed] (y) to (z);
\end{tikzpicture}}
  \subcaptionbox{\label{fig:cfp123-3}}{\begin{tikzpicture}
\input{./fig/pattern/template_fp4}
\draw [thick] (x) to [out=60, in=120] (w);
\draw [thick,dashed] (x) to [out=60, in=120] (z);
\draw [thick] (y) to (z);
\end{tikzpicture}}
  \subcaptionbox{\label{fig:cfp123-4}}{\begin{tikzpicture}
\input{./fig/pattern/template_fp4}
\draw [thick] (x) to [out=60, in=120] (w);
\draw [thick,dashed] (y) to [out=60, in=120] (w);
\draw [thick] (y) to (z);
\end{tikzpicture}}
  \caption{Complements of the forbidden patterns in Figure~\ref{fig:fp123}. }
  \label{fig:cfp123}
\end{figure}

A graph $G$ is an \emph{MPT graph} if 
each vertex $v$ can be assigned to an interval $I_v$ on the real line 
together with a point $p_v \in I_v$ 
such that $uv \in E(G) \iff \{p_u, p_v\} \subseteq I_u \cap I_v$ 
for any two vertices $u, v \in V(G)$. 
We call the set $\{(I_v, p_v) \colon\ v \in V(G)\}$ an \emph{MPT model} of $G$. 
The forbidden pattern characterization is known for MPT graphs. 
\begin{theorem}[\cite{CCFHHHS17-DAM,Hixon13-Master,SC15-DMTCS}]\label{thm:fp_MPT}
A graph is an MPT graph if and only if 
it has a vertex ordering 
that does not contain the pattern in Figure~\ref{fig:cfp123}\subref{fig:cfp123-2}. 
\end{theorem}

MPT graphs can be characterized as intersection graphs of axis-aligned $L$-shapes. 
An \emph{$L$-shape} is a union of a vertical line segment and a horizontal line segment in which 
the bottom endpoint of the vertical line segment coincides with the left endpoint of the horizontal one. 
The common endpoint is called the \emph{corner} of the $L$-shape. 
Consider a line $\mathcal{L}: y = -x$ in the $xy$-plane. 
A \emph{hook} is an $L$-shape whose corner lies on $\mathcal{L}$. 
A graph is an MPT graph if and only if it is the intersection graph of hooks~\cite{CCFHHHS17-DAM,SC15-DMTCS}. 
We call the collection of hooks a \emph{hook model} of the graph. 
\par
The transformation between MPT models and hook models can be described as follows. 
Let $\{([\ell_v, r_v], p_v) \colon\ v \in V(G)\}$ be an MPT model of a graph $G$, 
where $\ell_v$ and $r_v$ denote the left and right endpoints of the interval $I_v$, respectively. 
For each vertex $v$, we consider a hook where 
the corner is $(p_v, -p_v)$, 
the top endpoint is $(p_v, -\ell_v)$, and 
the right endpoint is $(r_v, -p_v)$; 
see Figure~\ref{fig:hook model}. 
The collection of such hooks forms a hook model of $G$~\cite{CCFHHHS17-DAM,SC15-DMTCS}. 

\begin{figure}[ht]
  \centering\begin{tikzpicture}
\def\len{0.5}
\useasboundingbox (-3.5*\len-0.5, -5*\len-0.1) rectangle (5*\len+0.5, 4.5*\len+0.3);
\tikzstyle{every node}=[inner sep=1pt]
\node [] (c1) at (1*\len, -1*\len) {};
\node [] (c2) at (2*\len, -2*\len) {};
\node [] (c3) at (3*\len, -3*\len) {};
\node [] (t1) at ($(c1) + (0, 3*\len)$) {};
\node [] (t3) at ($(c3) + (0, 3*\len)$) {};
\node [] (r1) at ($(c1) + (3*\len, 0)$) {};
\node [] (r2) at ($(c2) + (3*\len, 0)$) {};
\node [label=above left:$\mathcal{L}$] (L) at (-0.5*\len, 0.5*\len) {};
\draw [ultra thick] (L) -- ($(L) + (5*\len, -5*\len)$);
\draw [] 
	(t1) -- (c1) -- (r1)
	(c2) -- (r2)
	(t3) -- (c3)
;
\tikzstyle{every node}=[draw,rectangle,fill=black,minimum size=5pt,inner sep=2pt]
\node [label=above:$p_1$] (ap1) at (1*\len, 4.5*\len) {};
\node [label=above:$p_2$] (ap2) at (2*\len, 3.5*\len) {};
\node [label=above:$p_3$] (ap3) at (3*\len, 2.5*\len) {};
\draw [{|-|}] ($(ap1) - (3*\len, 0)$) -- ($(ap1) + (3*\len, 0)$);
\draw [{-|}] (ap2) -- ($(ap2) + (3*\len, 0)$);
\draw [{-|}] (ap3) -- ($(ap3) - (3*\len, 0)$);
\tikzstyle{every node}=[inner sep=2pt]
\node [label=above:$\ell_1$] (al1) at ($(ap1) - (3*\len, 0)$) {};
\node [label=above:$r_1$]    (ar1) at ($(ap1) + (3*\len, 0)$) {};
\node [label=above:$r_2$]    (ar2) at ($(ap2) + (3*\len, 0)$) {};
\node [label=above:$\ell_3$] (al3) at ($(ap3) - (3*\len, 0)$) {};
\draw [dotted] 
	(ap1) -- (t1)
	(ap2) -- (c2)
	(ap3) -- (t3)
	(ar1) -- (r1)
	(ar2) -- (r2)
;
\tikzstyle{every node}=[draw,rectangle,fill=black,minimum size=5pt,inner sep=1pt]
\node [label=left :$p_1$] (lp1) at (-3.5*\len, -1*\len) {};
\node [label=left :$p_2$] (lp2) at (-2.5*\len, -2*\len) {};
\node [label=left :$p_3$] (lp3) at (-1.5*\len, -3*\len) {};
\draw [{|-|}] ($(lp1) - (0, 3*\len)$) -- ($(lp1) + (0, 3*\len)$);
\draw [{-|}] (lp2) -- ($(lp2) - (0, 3*\len)$);
\draw [{-|}] (lp3) -- ($(lp3) + (0, 3*\len)$);
\tikzstyle{every node}=[inner sep=2pt]
\node [label=left :$\ell_1$] (ll1) at ($(lp1) + (0, 3*\len)$) {};
\node [label=left :$r_1$]    (lr1) at ($(lp1) - (0, 3*\len)$) {};
\node [label=left :$r_2$]    (lr2) at ($(lp2) - (0, 3*\len)$) {};
\node [label=left :$\ell_3$] (ll3) at ($(lp3) + (0, 3*\len)$) {};
\draw [dotted] 
	(lp1) -- (c1)
	(lp2) -- (c2)
	(lp3) -- (c3)
	(ll1) -- (t1)
	(ll3) -- (t3)
;
\end{tikzpicture}
  \caption{Illustration of the transformation between an MPT model and a hook model. }
  \label{fig:hook model}
\end{figure}
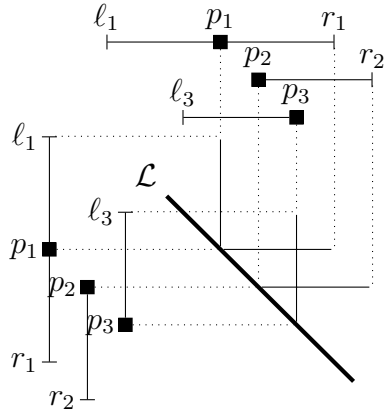

A hook is called a \emph{stick}~\cite{LHKLM19-TCS} 
if it degenerates to a horizontal or vertical line segment. 
In other words, a stick is a horizontal or vertical line segment 
whose left or bottom endpoint lies on the line $\mathcal{L}$. 
By abuse of notion, we refer to the endpoint on $\mathcal{L}$ as the \emph{corner} of the stick. 

We say that a hook is of \emph{unit length} if 
both horizontal and vertical line segments have unit length. 
We consider the intersection graphs of unit-length hooks and sticks, which in turn 
are precisely complements of $123$-representable graphs. 
Note that the right and top endpoints of the hooks and sticks lie on the line $\mathcal{L'}: y = -x + 1$. 
This indicates that 
the intersection graphs of unit-length hooks and sticks are simple-triangle graphs; 
see Figure~\ref{fig:ex 123}\subref{fig:model cricket}. 

\begin{figure}[h]
  \begin{tabular}{ccc}
    \begin{minipage}[]{0.6\hsize}
      \centering
      \subcaptionbox{\label{fig:co-cricket}}{\begin{tikzpicture}
\def\len{1.0}
\useasboundingbox (-0.7, -0.3) rectangle (2.0*\len+0.7, 1*\len+0.3);
\tikzstyle{every node}=[draw,circle,fill=white,minimum size=5pt,inner sep=0pt]
\node [label=above:$1$] (v1) at (0, 1*\len) {};
\node [label=above:$3$] (v3) at (1*\len, 1*\len) {};
\node [label=below:$2$] (v2) at (1*\len, 0) {};
\node [label=below:$4$]       (v4) at (0, 0) {};
\node [label=above:$5$]       (v5) at (2*\len, 0.5*\len) {};
\draw [] (v3) -- (v1) -- (v4) -- (v3) -- (v2) -- (v4);
\end{tikzpicture}}
      \subcaptionbox{\label{fig:cricket}}{\begin{tikzpicture}
\def\len{1.0}
\useasboundingbox (-0.7, -0.3) rectangle (2.0*\len+0.7, 1*\len+0.3);
\tikzstyle{every node}=[draw,circle,fill=white,minimum size=5pt,inner sep=0pt]
\node [label=below:$3$] (a) at (0, 0) {};
\node [label=below:$5$] (b) at (1*\len, 0) {};
\node [label=below:$4$] (c) at (2*\len, 0) {};
\node [label=above:$1$] (d) at ($(b) + (120:\len)$) {};
\node [label=above:$2$] (e) at ($(b) + ( 60:\len)$) {};
\draw [] (a) -- (b) -- (d) -- (e) -- (b) -- (c);
\end{tikzpicture}}
\\
      \subcaptionbox{\label{fig:construction}}{\begin{tikzpicture}
\def\len{1.2}
\def\lenh{0.5}
\useasboundingbox (-1*\len, -6*\lenh-0.2) rectangle (5*\len, 0.8);
\tikzstyle{every node}=[draw,circle,fill=white,minimum size=5pt,inner sep=0pt]
\node [label=below:$1$] (v1) at (0*\len, 0) {};
\node [label=below:$2$] (v2) at (1*\len, 0) {};
\node [label=below:$3$] (v3) at (2*\len, 0) {};
\node [label=below:$4$] (v4) at (3*\len, 0) {};
\node [label=below:$5$] (v5) at (4*\len, 0) {};
\draw [] (v1) to (v2);
\draw [] (v1) to [out=30, in=150] (v5);
\draw [] (v2) to [out=30, in=150] (v5);
\draw [] (v3) to [out=30, in=150] (v5);
\draw [] (v4) to (v5);
\tikzstyle{every node}=[draw,rectangle,fill=black,minimum size=5pt,inner sep=2pt]
\node [label=above:$p_1$] (p1) at (0*\len, -2*\lenh) {};
\node [label=above:$p_2$] (p2) at (1*\len, -3*\lenh) {};
\node [label=above:$p_3$] (p3) at (2*\len, -4*\lenh) {};
\node [label=above:$p_4$] (p4) at (3*\len, -5*\lenh) {};
\node [label=above:$p_5$] (p5) at (4*\len, -6*\lenh) {};
\draw [thick,{-|}] (p1) -- ($(p1) + (4*\len, 0) + (1*\len/6, 0)$);
\draw [thick,{-|}] (p2) -- ($(p2) + (3*\len, 0) + (2*\len/6, 0)$);
\draw [thick,{-|}] (p3) -- ($(p3) + (2*\len, 0) + (3*\len/6, 0)$);
\draw [thick,{-|}] (p4) -- ($(p4) + (1*\len, 0) + (4*\len/6, 0)$);
\draw [thick,{-|}] (p2) -- ($(p2) - (1*\len, 0) - (4*\len/6, 0)$);
\draw [thick,{-|}] (p5) -- ($(p5) - (4*\len, 0) - (1*\len/6, 0)$);
\tikzstyle{every node}=[inner sep=2pt]
\node [label=above:$r_1$] at ($(p1) + (4*\len, 0) + (1*\len/6, 0)$) {};
\node [label=above:$r_2$] at ($(p2) + (3*\len, 0) + (2*\len/6, 0)$) {};
\node [label=above:$r_3$] at ($(p3) + (2*\len, 0) + (3*\len/6, 0)$) {};
\node [label=above:$r_4$] at ($(p4) + (1*\len, 0) + (4*\len/6, 0)$) {};
\node [label=above:$\ell_2$] at ($(p2) - (1*\len, 0) - (4*\len/6, 0)$) {};
\node [label=above:$\ell_5$] at ($(p5) - (4*\len, 0) - (1*\len/6, 0)$) {};
\end{tikzpicture}}
    \end{minipage}
    \begin{minipage}[]{0.4\hsize}
      \centering
      \subcaptionbox{\label{fig:model cricket}}{\begin{tikzpicture}
\def\len{0.4}
\useasboundingbox (0*\len-0.2, -5.5*\len-0.2) rectangle (9*\len+0.5, 3.5*\len+0.5);
\tikzstyle{every node}=[inner sep=1pt]
\node [label=below left :$c_1$] (c1) at (1.0*\len, -1.0*\len) {};
\node [label=below left :$c_2$] (c2) at (2.0*\len, -2.0*\len) {};
\node [label=below left :$c_3$] (c3) at (3.0*\len, -3.0*\len) {};
\node [label=below left :$c_4$] (c4) at (4.0*\len, -4.0*\len) {};
\node [label=below left :$c_5$] (c5) at (5.0*\len, -5.0*\len) {};
\node [label=above right:$t_2$] (t2) at (2.0*\len,  3.0*\len) {};
\node [label=above right:$t_5$] (t5) at (5.0*\len,  0.0*\len) {};
\node [label=above right:$r_1$] (r1) at (6.0*\len, -1.0*\len) {};
\node [label=above right:$r_2$] (r2) at (7.0*\len, -2.0*\len) {};
\node [label=above right:$r_3$] (r3) at (8.0*\len, -3.0*\len) {};
\node [label=above right:$r_4$] (r4) at (9.0*\len, -4.0*\len) {};
\node [label=above left:$\mathcal{L}$] (L) at (0.5*\len, -0.5*\len) {};
\draw [ultra thick] (L) -- (5.5*\len, -5.5*\len);
\node [label=above left:$\mathcal{L'}$] (L') at (1.5*\len, 3.5*\len) {};
\draw [ultra thick] (L') -- (9.5*\len, -4.5*\len);
\draw [] 
	(c1) -- (r1)
	(t2) -- (c2) -- (r2)
	(c3) -- (r3)
	(c4) -- (r4)
	(t5) -- (c5)
;
\end{tikzpicture}}
    \end{minipage}
  \end{tabular}
  \caption{
	\subref{fig:co-cricket} A graph $G$. 
	\subref{fig:cricket} The complement $\overline{G}$ of $G$. 
	\subref{fig:construction} Illustration of the construction of an MPT model in the proof of Lemma~\ref{lem:123-2}. 
	\subref{fig:model cricket} The hook model of $\overline{G}$. 
	The vertices are labeled based on the corners on $\mathcal{L}$. 
	The graph $G$ is $12$-representable by $123$-avoiding word $w = 432152$ 
	obtained as shown in the proof of Theorem~\ref{thm:123}. 
  }
  \label{fig:ex 123}
\end{figure}
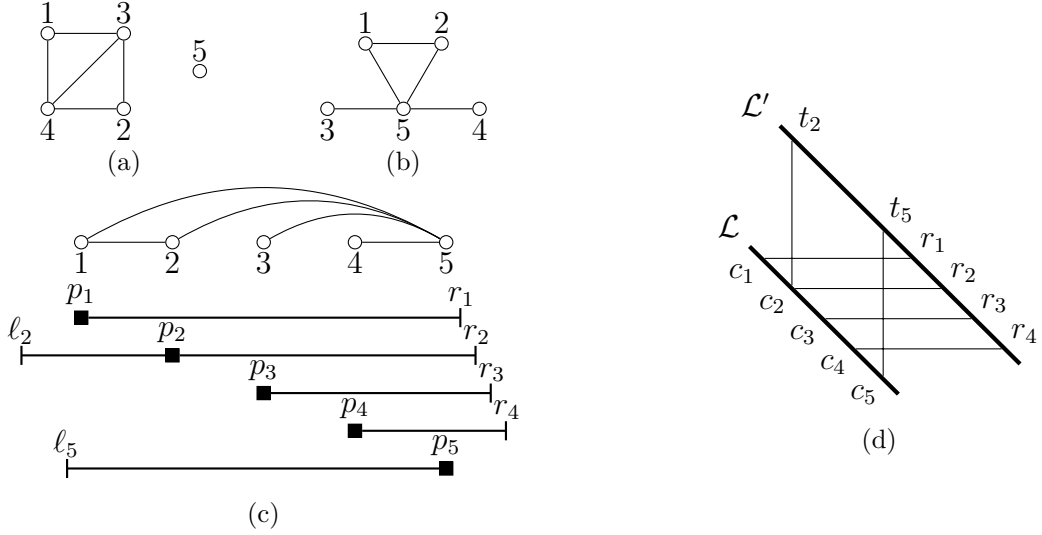

\begin{lemma}\label{lem:123-2}
If a labeled graph does not contain any pattern in Figure~\ref{fig:cfp123}, 
then it is the intersection graph of unit-length hooks and sticks. 
\end{lemma}
\begin{proof}
Let $G$ be a labeled graph containing no pattern in Figure~\ref{fig:cfp123}. 
First, we show that we can assume without loss of generality that 
$G$ contains no isolated vertices. 
Suppose that $G$ contains an isolated vertex $i$. 
Since $G$ does not contain the pattern in Figure~\ref{fig:cfp123}\subref{fig:cfp123-1}, 
there are no edges between vertices larger than $i$ and less than $i$. 
Thus, if two subgraphs induced by vertices larger than $i$ and less than $i$ 
are the intersection graphs of unit-length hooks and sticks, then 
$G$ is the intersection graph of unit-length hooks and sticks. 
\par
We next construct an MPT model of $G$. 
The construction is the same as that in~\cite{CCFHHHS17-DAM,SC15-DMTCS}. 
For each $i \in V(G)$, we define as follows. 
\begin{flalign*}
\ell_i & = \min\left\{i, j - \frac{n-i+1}{n+1}\right\}, \text{where $j$ is the smallest vertex with $ji \in E(G)$}. 
           \\
p_i    & =  i. 
           \\
r_i    & = \max\left\{i, j + \frac{i}{n+1}\right\}, \text{where $j$ is the largest vertex with $ij \in E(G)$}. 
\end{flalign*}
Figure~\ref{fig:ex 123}\subref{fig:construction} illustrates an example of the construction.
\begin{claim}
$\{([\ell_i, r_i], p_i) \colon\ i \in V(G)\}$ is an MPT model of $G$. 
\end{claim}
\begin{proof}[Proof of Claim]
The proof is the same as that in~\cite{CCFHHHS17-DAM,SC15-DMTCS}. 
Let $i$ and $j$ be two vertices with $i < j$. 
If $ij \in E(G)$ then $\ell_j < p_i$ and $p_j < r_i$ by construction; 
that is, $p_i, p_j \in I_i \cap I_j$, 
where $I_i = [\ell_i, r_i]$ denotes the interval of $i$ for each $i \in V(G)$. 
Suppose to the contrary that $ij \notin E(G)$ but $p_i, p_j \in I_i \cap I_j$. 
Since $ij \notin E(G)$ and $\ell_j < p_i$, 
there must be a vertex $i' < i$ such that $i'j \in E(G)$. 
Similarly, $ij \notin E(G)$ and $p_j < r_i$ indicate that 
there must be a vertex $j' > j$ such that $ij' \in E(G)$. 
However, now, the vertices $i', i, j, j'$ form the pattern in Figure~\ref{fig:cfp123}\subref{fig:cfp123-2}, 
a contradiction. 
\end{proof}
\par
We now show the properties the model satisfies. 
Let $V^+ = \{i \colon\ p_i \neq r_i\}$ and $V^- = \{i \colon\ p_i \neq \ell_i\}$. 
For example, in the graph in Figure~\ref{fig:ex 123}, 
$V^+ = \{1, 2, 3, 4\}$ and $V^- = \{2, 5\}$. 
Since we assume that $G$ contains no isolated vertices, $V^+ \cup V^- = V(G)$. 
For each $i \in V^+$, 
let $I_i^+ = [p_i, r_i]$ and $\mathcal{I}^+ = \{I_i^+ \colon\ i \in V^+\}$. 
Similarly, 
for each $i \in V^-$, 
let $I_i^- = [\ell_i, p_i]$ and $\mathcal{I}^- = \{I_i^- \colon\ i \in V^-\}$. 
Moreover, let $\mathcal{I} = \mathcal{I}^+ \cup \mathcal{I}^-$. 
\begin{claim}
No two intervals of $\mathcal{I}$ contain each other. 
\end{claim}
\begin{proof}[Proof of Claim]
Suppose first that an interval $I_i^+ \in \mathcal{I}^+$ contains 
another interval $I_j^+ \in \mathcal{I}^+$. 
Then, $i < j$. 
There must be two vertices $k$ and ${\ell}$ with $j < k < \ell$ such that 
$i{\ell}, jk \in E(G)$ and $j{\ell} \notin E(G)$. 
However, the vertices $i, j, k, {\ell}$ form a pattern in Figure~\ref{fig:cfp123}\subref{fig:cfp123-4}, 
a contradiction. 
Thus, $I_i^+$ contains no interval of $\mathcal{I}^+$. 
Suppose now that $I_i^+$ contains an interval $I_j^- \in \mathcal{I}^-$. 
Then, $i < j$. 
There must be two vertices $k$ and ${\ell}$ with $i < k < j < \ell$ such that 
$i{\ell}, jk \in E(G)$ and $ji \notin E(G)$. 
However, the vertices $i, k, j, {\ell}$ form a pattern in Figure~\ref{fig:cfp123}\subref{fig:cfp123-3}, 
a contradiction. 
Therefore, $I_i^+$ contains no interval of $\mathcal{I}$. 
Similar arguments would show that no interval of $\mathcal{I}^-$ contains another interval of $\mathcal{I}$. 
\end{proof}
\par
Finally, we adjust all the intervals of $\mathcal{I}$ to length $1$ 
without changing the order of the points of 
$\{\ell_i, p_i, r_i \colon\ i \in V(G)\}$. 
In other words, we change the position of the points so that 
for each vertex $i$, either 
$\ell_i = p_i = r_i - 1$, 
$\ell_i + 1 = p_i = r_i - 1$, or 
$\ell_i + 1 = p_i = r_i$ 
holds. 
Such an MPT model can be transformed into a hook model 
consisting of unit-length hooks and sticks 
by the transformation described above. 
\par
The adjustment can be performed 
in the same way as we obtain a unit interval model from a proper one~\cite{BW99-DM}. 
We process the intervals of $\mathcal{I}$ from left to right. 
Let $I_i^* = [a, b]$ be the unadjusted interval with the leftmost left endpoint. 
When $I_i^*$ does not contain the right endpoints of any other intervals, 
shrink or expand $[a, b]$ to $[a, a+1]$ and 
translate $[b, \infty)$ to $[a+1, \infty)$. 
Now, $I_i^*$ has length $1$. 
The order of points does not change, and 
the adjusted intervals still have length $1$ because no interval contains $I_i^*$. 
When $I_i^*$ contains the right endpoint of some other interval, 
we define $\alpha$ as the rightmost such right endpoint. 
The point $\alpha$ is the right endpoint of an adjusted interval, 
and hence, $\alpha < \min\{a+1, b\}$. 
Now, shrink or expand $[\alpha, b]$ to $[\alpha, a+1]$ and 
translate $[b, \infty)$ to $[a+1, \infty)$. 
Performing the same procedure for each interval from left to right yields the required model. 
\end{proof}

\begin{lemma}\label{lem:123-3}
If a graph is the intersection graph of unit-length hooks and sticks, then 
its complement is $12$-representable by a $123$-avoiding word. 
\end{lemma}
\begin{proof}
Consider a hook model of the graph $G$. 
Without loss of generality, we assume that 
all the corners are distinct. 
We assign a label $x$ to a vertex $v$ if the corner of $v$ is the $x$th one 
among all corners on $\mathcal{L}$ from top-left to bottom-right. 
Recall that the right and top endpoints of the hooks and sticks 
lie on the line $\mathcal{L'}$. 
We form a word $w$ by the endpoints on $\mathcal{L'}$ so that 
the $x$th letter of $w$ is the label of a vertex $v$
if, among all endpoints on $\mathcal{L'}$ 
(i.e., both right and top endpoints) \emph{from bottom-right to top-left}, 
the $x$th endpoint is of the hook or stick of $v$. 
\par
Figure~\ref{fig:ex 123} illustrates an example of the construction. 
The graph in Figure~\ref{fig:ex 123}\subref{fig:cricket} is 
the intersection graph of unit-length hooks and sticks 
in Figure~\ref{fig:ex 123}\subref{fig:model cricket}. 
The vertices are labeled based on the corners on $\mathcal{L}$. 
By reading the labels of the endpoints on $\mathcal{L'}$ 
\emph{from bottom-right to top-left}, we obtain the word $w = 432151$, 
which is a $12$-representant of
the graph in Figure~\ref{fig:ex 123}\subref{fig:co-cricket}. 
\par
We first claim that the word $w$ is a $12$-representant of the complement $\overline{G}$ of $G$. 
Let $c_x$, $r_x$, and $t_x$ denote the corner, right, and top endpoints of the hook or stick of a vertex (labeled with) $x$, 
respectively. 
Let $x$ and $y$ be two vertices with $x < y$. 
Consider the case when both vertices correspond to unit-length hooks; 
the other cases can be proven in the same way. 
By definition, $c_x$ lies to the left of $c_y$ on $\mathcal{L}$. 
Then, 
the hooks of $x$ and $y$ do not intersect 
if and only if 
both $r_x$ and $t_x$ on $\mathcal{L'}$ lie to the left of both $r_y$ and $t_y$ 
if and only if 
$w_{\{x, y\}} = yyxx$. 
Thus, $w$ is a $12$-representant of $\overline{G}$. 
\par
We now claim that the word $w = w_1 w_2 \ldots w_{\ell}$ is $123$-avoiding. 
Suppose to the contrary that there is a triple of indices $i < j < k$ such that $w_i < w_j < w_k$. 
Let $w_i = x$, $w_j = y$, and $w_k = z$, that is, $x < y < z$. 
Since $c_x$ lies to the left of $c_z$, 
the occurrences $w_i$ and $w_k$ correspond to $r_x$ and $t_z$, respectively. 
Since $c_y$ lies between $c_x$ and $c_z$, 
neither $r_y$ nor $t_y$ lies between $r_x$ and $t_z$, a contradiction. 
Thus, $w$ is $123$-avoiding. 
\end{proof}

The lemmas above indicate the main theorem of this section. 
\begin{theorem}\label{thm:123}
The following statements are equivalent for a graph $G$: 
\begin{enumerate}[label=\textup{(\roman*)}]
\item $G$ is $12$-representable by a $123$-avoiding word. 
\item $G$ has a vertex ordering that does not contain any pattern in Figure~\ref{fig:fp123}. 
\item $G$ is the complement of the intersection graph of unit-length hooks and sticks. 
\end{enumerate}
\end{theorem}
\begin{proof}
The implications (i) $\implies$ (ii), (ii) $\implies$ (iii), and (iii) $\implies$ (i) 
follow from Lemmas~\ref{lem:123-1},~\ref{lem:123-2}, and~\ref{lem:123-3}, 
respectively. 
\end{proof}

Theorems~\ref{thm:123} and~\ref{thm:fp_MPT} indicate that 
the class of $123$-representable graphs is a subclass of the complements of MPT graphs. 
\begin{corollary}
A graph is the complement of an MPT graph 
if it is $12$-representable by a $123$-avoiding word. 
\end{corollary}

Furthermore, 
complements of bipartite permutation graphs and unit interval graphs are 
subclasses of $123$-representable graphs. 
\begin{corollary}
The complements of bipartite permutation graphs 
are $12$-representable by $123$-avoiding words. 
\end{corollary}
\begin{proof}
Let $G$ be a bipartite permutation graph. 
Theorem~\ref{thm:321} indicates that $G$ is $321$-representable. 
By Lemma~\ref{lem:321:permutation}, 
there is a $321$-avoiding permutation $\pi$ representing $G$. 
The reverse $r(\pi)$ of $\pi$ is $123$-avoiding and 
$12$-representing the complement $\overline{G}$ of $G$. 
\end{proof}

\begin{corollary}
The complements of unit interval graphs 
are $12$-representable by $123$-avoiding words. 
\end{corollary}
\begin{proof}
Let $G$ be a unit interval graph. 
By Theorem~\ref{thm:123}, 
it suffices to show that $G$ is the intersection graph of unit-length hooks. 
Consider an interval model of $G$ lying on the line $\mathcal{L'}: y = -x + 1$ 
in which each interval is of length $\sqrt{2}$. 
For each interval, we can form a hook whose top and right endpoints are 
the left and right endpoints of the interval, respectively. 
The collection of such hooks is a hook model of $G$. 
\end{proof}

\section{$132$-representable graphs}\label{sec:132}
In this section, we show a forbidden pattern characterization 
and the related results for $132$-representable graphs. 
\begin{lemma}\label{lem:132_1}
If a labeled graph contains a pattern in Figure~\ref{fig:fp132}, 
then it is not $12$-representable by a $132$-avoiding word. 
\end{lemma}
\begin{proof}
By Proposition~\ref{prop:labeled induced subgraph}, 
it suffices to show that any pattern in Figure~\ref{fig:fp132} 
is not $12$-representable by a $132$-avoiding word. 
Let $H_1$ be a graph with 
$V(H_1) = \{1, 2, 3\}$ such that $13 \notin E(H_1)$ and $23 \in E(H_1)$. 
Suppose that there is a $12$-representant $w_1$ of $H_1$. 
Since $23 \in E(H_1)$, 
the rightmost $3$ occurs before the leftmost $2$ in $w_1$. 
Since $13 \notin E(H_1)$, some $1$ occurs before the rightmost $3$. 
Thus, the pattern $132$ occurs in $w_1$. 
Hence, the pattern in Figure~\ref{fig:fp132}\subref{fig:fp132-1} 
is not $12$-representable by a $132$-avoiding word. 
\par
Let $H_2$ be a graph with 
$V(H_2) = \{1, 2, 3, 4\}$ such that $14 \in E(H_2)$ and $13, 24 \notin E(H_2)$. 
Suppose that there is a $12$-representant $w_2$ of $H_2$. 
Since $14 \in E(H_2)$, 
the rightmost $4$ occurs before the leftmost $1$ in $w_2$. 
Since $13 \notin E(H_2)$, some $3$ occurs after the leftmost $1$. 
Since $24 \notin E(H_2)$, some $2$ occurs before the rightmost $4$. 
Thus, the pattern $2413$ occurs in $w_2$. 
Hence, $w_2$ contains the pattern $132$. 
Therefore, the pattern in Figure~\ref{fig:fp132}\subref{fig:fp132-2} 
is not $12$-representable by a $132$-avoiding word. 
\end{proof}

\begin{figure}[ht]
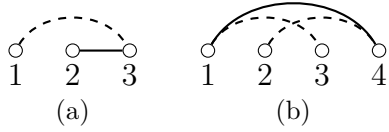

  \centering
  \subcaptionbox{\label{fig:fp132-1}}{\begin{tikzpicture}
\input{./fig/pattern/template_fp3}
\draw [thick] (y) -- (z);
\draw [thick,dashed] (x) to [out=60, in=120] (z);
\end{tikzpicture}}
  \subcaptionbox{\label{fig:fp132-2}}{\begin{tikzpicture}
\input{./fig/pattern/template_fp4}
\draw [thick,dashed] (x) to [out=60, in=120] (z);
\draw [thick,dashed] (y) to [out=60, in=120] (w);
\draw [thick] (x) to [out=60, in=120] (w);
\end{tikzpicture}}
  \caption{Forbidden patterns of $132$-representable graphs. }
  \label{fig:fp132}
\end{figure}

\begin{lemma}\label{lem:132_2}
If a labeled graph does not contain any pattern in Figure~\ref{fig:fp132}, 
then it is $12$-representable by a $132$-avoiding word. 
\end{lemma}
\begin{proof}
Let $G$ be a labeled graph containing no pattern in Figure~\ref{fig:fp132}. 
By Theorem~\ref{thm:fp_m_int}, the graph $G$ is a co-interval graph. 
We first construct an interval model of the complement $\overline{G}$ of $G$. 
Since $G$ contains no pattern in Figure~\ref{fig:fp132}\subref{fig:fp132-1}, 
for each vertex $i$ of $G$, there is an index $\ell'_i \leq i$ such that 
$i$ is adjacent to every vertex in $\{1, 2, \ldots, \ell'_i - 1\}$ and 
non-adjacent to any vertex in $\{\ell'_i, \ell'_i+1, \ldots, i-1\}$. 
Formally, 
\begin{flalign*}
\ell'_i & = \min\left\{i, j\right\}, \text{where $j$ is the smallest vertex with $ji \notin E(G)$}. 
\end{flalign*}
Let 
$\ell_i = \ell'_i - \frac{i}{n+1}$, 
$r_i = i$, and
$I_i = [\ell_i, r_i]$. 
The collection $\{I_i \colon\ i \in V(G)\}$ is an interval model of $\overline{G}$. 
\par
Figure~\ref{fig:co-X58} illustrates an example of the construction. 
We can check that the graph in Figure~\ref{fig:co-X58} 
contains no pattern in Figure~\ref{fig:fp132}. 
We have $\ell'_1 = 1$, $\ell'_2 = 1$, $\ell'_3 = 2$, $\ell'_4 = 4$, $\ell'_5 = 2$, and $\ell'_6 = 3$. 
The obtained model is shown in Figure~\ref{fig:co-X58}\subref{fig:co-X58:model}. 

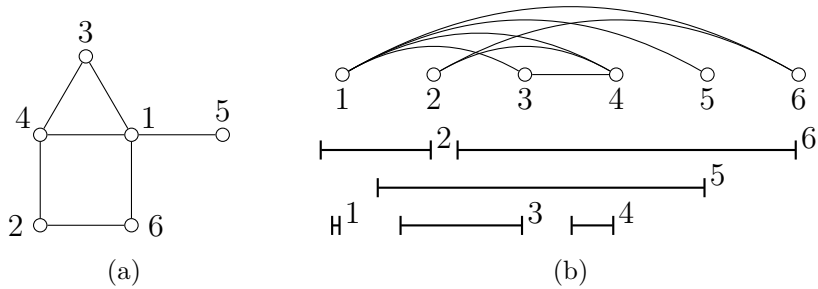
\begin{figure}[ht]
  \centering
  \subcaptionbox{\label{fig:co-X58:graph}}{\begin{tikzpicture}
\def\len{1.2}
\useasboundingbox (-0.7, -0.2) rectangle (2*\len+0.5, 2*\len+0.4);
\tikzstyle{every node}=[draw,circle,fill=white,minimum size=5pt,inner sep=1pt]
\node [label=left:$2$]        (a) at (0, 0) {};
\node [label=right:$6$]       (b) at (1*\len, 0) {};
\node [label=above left:$4$]  (c) at (0, 1*\len) {};
\node [label=above right:$1$] (d) at (1*\len, 1*\len) {};
\node [label=above:$3$]       (e) at ($(c) + (60:\len)$) {};
\node [label=above:$5$]       (f) at (2*\len, 1*\len) {};
\draw [] (c) -- (a) -- (b) -- (d) -- (e) -- (c) -- (d) -- (f);
\end{tikzpicture}}
  \subcaptionbox{\label{fig:co-X58:model}}{\begin{tikzpicture}
\def\len{1.2}
\def\la{-1.0}
\def\lb{-1.5}
\def\lc{-2.0}
\useasboundingbox (-0.8, -2.2) rectangle (5*\len+0.8, 1.0);
\tikzstyle{every node}=[draw,circle,fill=white,minimum size=5pt,inner sep=1pt]
\node [label=below:$1$] (v1) at (0*\len, 0) {};
\node [label=below:$2$] (v2) at (1*\len, 0) {};
\node [label=below:$3$] (v3) at (2*\len, 0) {};
\node [label=below:$4$] (v4) at (3*\len, 0) {};
\node [label=below:$5$] (v5) at (4*\len, 0) {};
\node [label=below:$6$] (v6) at (5*\len, 0) {};
\draw [] (v1) to [out=30, in=150] (v3);
\draw [] (v1) to [out=30, in=150] (v4);
\draw [] (v1) to [out=30, in=150] (v5);
\draw [] (v1) to [out=30, in=150] (v6);
\draw [] (v2) to [out=30, in=150] (v4);
\draw [] (v2) to [out=30, in=150] (v6);
\draw [] (v3) to (v4);
\tikzstyle{every node}=[minimum size=1pt,inner sep=0pt]
\node [label=above right:$1$] (a) at (0*\len, \lc) {};
\draw [thick,{|-|}] (a) -- ($(a) - (0*\len, 0) - (1*\len/8, 0)$);
\node [label=above right:$2$] (b) at (1*\len, \la) {};
\draw [thick,{|-|}] (b) -- ($(b) - (1*\len, 0) - (2*\len/8, 0)$);
\node [label=above right:$3$] (c) at (2*\len, \lc) {};
\draw [thick,{|-|}] (c) -- ($(c) - (1*\len, 0) - (3*\len/8, 0)$);
\node [label=above right:$4$] (d) at (3*\len, \lc) {};
\draw [thick,{|-|}] (d) -- ($(d) - (0*\len, 0) - (4*\len/8, 0)$);
\node [label=above right:$5$] (e) at (4*\len, \lb) {};
\draw [thick,{|-|}] (e) -- ($(e) - (3*\len, 0) - (5*\len/8, 0)$);
\node [label=above right:$6$] (f) at (5*\len, \la) {};
\draw [thick,{|-|}] (f) -- ($(f) - (3*\len, 0) - (6*\len/8, 0)$);
\end{tikzpicture}}
  \caption{
    \subref{fig:co-X58:graph} A $132$-representable graph. 
    \subref{fig:co-X58:model} Illustration of the construction of the interval model in the proof of Lemma~\ref{lem:132_2}. 
  }
  \label{fig:co-X58}
\end{figure}

Let $w$ be a word obtained by reading the indices of 
the endpoints of the intervals 
(i.e., both left and right endpoints) \emph{from right to left}. 
For example, from the graph in Figure~\ref{fig:co-X58}, 
we obtain the word $w = 654436235112$. 
Since all endpoints are distinct, $w$ is uniquely defined. 
For any two vertices $i$ and $j$ with $i < j$, 
we have $w_{\{i, j\}} = jjii$ if and only if $I_i$ lies entirely to the left of $I_j$ 
if and only if $ij \in E(G)$. 
Thus, $w$ is a $12$-representant of $G$. 
\par
We claim that $w = w_1 w_2 \ldots w_{2n}$ is $132$-avoiding. 
Suppose to the contrary that there is a triple of indices $i < j < k$ such that $w_i < w_k < w_j$. 
Let $w_i = a$, $w_j = c$, and $w_k = b$, that is, $a < b < c$. 
The occurrence $w_j$ corresponds to either $\ell_c$ or $r_c$. 
If $w_j$ corresponds to $r_c$, then $\ell_a < r_a = a < c = r_c$ indicates that 
letter $a$ cannot occur before $w_j$ in $w$, a contradiction. 
Thus, $w_j$ corresponds to $\ell_c$. 
Since $w_i$ occurs before $w_j$ in $w$, we have $\ell_c < r_a$. 
Thus, $r_a = a < c = r_c$ indicates that $I_c$ intersects $I_a$, that is, $ac \notin E(G)$. 
On the other hand, $w_k$ corresponds to either $\ell_b$ or $r_b$. 
If $w_k$ corresponds to $r_b$, then $\ell_a < r_a = a < b = r_b$ indicates that 
letter $a$ cannot occur before $w_k$ in $w$, a contradiction. 
Thus, $w_k$ corresponds to $\ell_b$. 
Now, $j < k$ indicates $\ell_b < \ell_c$. 
If $\ell'_b = \ell'_c$ then $b < c$ implies $\ell_b > \ell_c$, a contradiction. 
(For example, in Figure~\ref{fig:co-X58}, $\ell'_3 = \ell'_5= 2$ and $\ell_3 > \ell_5$.) 
Thus, $\ell'_b < \ell'_c$; that is, 
there is a vertex $d$ such that $\ell'_b = d < \ell'_c$. 
Then, $d \in N(c) \setminus N(b)$ and $d = r_d < \ell_c < r_a = a$. 
Therefore, 
the vertices $d, a, b, c$ form a pattern in Figure~\ref{fig:fp132}\subref{fig:fp132-2}, a contradiction. 
Thus, $w$ is $132$-avoiding. 
\end{proof}

Lemmas~\ref{lem:132_1} and~\ref{lem:132_2} yield the main theorem of this section. 
\begin{theorem}\label{thm:132}
A labeled graph is $12$-representable by a $132$-avoiding word if and only if 
it does not contain any pattern in Figure~\ref{fig:fp132}. 
\end{theorem}

Lemma~\ref{lem:132_2} and Theorem~\ref{thm:fp_m_int} indicate that 
the class of $132$-representable graphs is a subclass of co-interval graphs. 
\begin{corollary}\label{cor:132:co-interval}
A graph is a co-interval graph if it is $12$-representable by a $132$-avoiding word. 
\end{corollary}

Furthermore, 
complements of trivially perfect graphs are 
a subclass of $132$-representable graphs. 
\begin{corollary}\label{prop:132 TPG}
The complements of trivially perfect graphs 
are $12$-representable by $123$-avoiding words. 
\end{corollary}
\begin{proof}
Let $G$ be a trivially perfect graph. 
Theorem~\ref{thm:231} indicates that $G$ is $231$-representable. 
By Lemma~\ref{lem:231}, 
there is a $231$-avoiding permutation $\pi$ representing $G$. 
The reverse $r(\pi)$ of $\pi$ is $132$-avoiding and 
$12$-representing the complement $\overline{G}$ of $G$. 
\end{proof}

\begin{remark}
The inclusion of Corollary~\ref{cor:132:co-interval} is proper 
because twin-house in Figure~\ref{fig:twin-house}\subref{fig:twin-house:graph} 
is a co-interval graph, whose model is illustrated in 
Figure~\ref{fig:twin-house}\subref{fig:twin-house:model}, 
but is not $12$-representable by $132$-avoiding words. 
To see this, we prove by contradiction that 
any labeling contains one of the forbidden patterns 
in Figure~\ref{fig:fp132}. 
\par
If the label $1$ is assigned to $a$, $b$, $d$, or $e$, 
then the pattern in Figure~\ref{fig:fp132}\subref{fig:fp132-1} occurs; 
for example, if $1$ is assigned to $a$, then the vertices $a, c, d$ form the pattern 
since the labels of $c$ and $d$ would be larger than $1$. 
Thus, the label $1$ must be assigned to $c$ or $f$. 
Without loss of generality, we assume that $1$ is assigned to $f$. 
\par
If the label $2$ is assigned to $a$ or $c$, 
then the pattern in Figure~\ref{fig:fp132}\subref{fig:fp132-1} occurs 
since the label of $b$ would be larger than $2$. 
If the label $2$ is assigned to $d$ or $e$, 
then the pattern in Figure~\ref{fig:fp132}\subref{fig:fp132-1} occurs 
since the labels of $a$ and $b$ would be larger than $2$. 
Thus, the label $2$ must be assigned to $b$. 
\par
If the label $3$ is assigned to $a$, 
then the vertices $a, c, d$ form the pattern in Figure~\ref{fig:fp132}\subref{fig:fp132-1}. 
If the label $3$ is assigned to $c$, 
then the vertices $b, c, d$ form the pattern in Figure~\ref{fig:fp132}\subref{fig:fp132-1}. 
Thus, the label $3$ must be assigned to $d$ or $e$. 
Without loss of generality, we assume that $3$ is assigned to $e$. 
\par
If the label $4$ is assigned to $a$, 
then the vertices $a, c, d$ form the pattern in Figure~\ref{fig:fp132}\subref{fig:fp132-1}. 
If the label $4$ is assigned to $c$, 
then the vertices $b, c, d$ form the pattern in Figure~\ref{fig:fp132}\subref{fig:fp132-1}. 
Thus, the label $4$ must be assigned to $d$. 
However, now, the vertices $b, e, d, a$ form the pattern in Figure~\ref{fig:fp132}\subref{fig:fp132-2}. 
\end{remark}

\begin{figure}[ht]
  \centering
  \subcaptionbox{\label{fig:twin-house:graph}}{\begin{tikzpicture}
\def\len{1.0}
\useasboundingbox (-0.8, -0.1) rectangle (1*\len+0.8, 2*\len+0.1);
\tikzstyle{every node}=[draw,circle,fill=white,minimum size=5pt,inner sep=1pt]
\node [label=left:$a$ ] (a) at (0, 0) {};
\node [label=right:$b$] (b) at (1*\len, 0) {};
\node [label=right:$c$] (c) at (1*\len, 1*\len) {};
\node [label=right:$d$] (d) at (1*\len, 2*\len) {};
\node [label=left:$e$ ] (e) at (0, 2*\len) {};
\node [label=left:$f$ ] (f) at (0, 1*\len) {};
\draw [] (a) -- (b) -- (c) -- (d) -- (f) -- (c) -- (e) -- (f) -- (a);
\end{tikzpicture}}
  \subcaptionbox{\label{fig:twin-house:model}}{\begin{tikzpicture}
\def\len{0.6}
\def\la{0}
\def\lb{-0.5}
\def\lc{-1.0}
\def\ld{-1.5}
\useasboundingbox (-1.5*\len-0.5, -1.7) rectangle (8.5*\len+0.5, 0.5);
\tikzstyle{every node}=[inner sep=2pt]
\node [label=above:$a$] (a) at (1.5*\len, \la) {};
\draw [very thick,{|-|}] ($(a) + (-1.5*\len, 0)$) -- ($(a) + (1.5*\len, 0)$);
\node [label=above:$b$] (b) at (5.5*\len, \la) {};
\draw [very thick,{|-|}] ($(b) + (-1.5*\len, 0)$) -- ($(b) + (1.5*\len, 0)$);
\node [label=above:$c$] (c) at (0.0*\len, \lc) {};
\draw [very thick,{|-|}] ($(c) + (-1.5*\len, 0)$) -- ($(c) + (1.5*\len, 0)$);
\node [label=above:$d$] (d) at (3.5*\len, \lb) {};
\draw [very thick,{|-|}] ($(d) + (-1.5*\len, 0)$) -- ($(d) + (1.5*\len, 0)$);
\node [label=above:$e$] (e) at (3.5*\len, \ld) {};
\draw [very thick,{|-|}] ($(e) + (-1.5*\len, 0)$) -- ($(e) + (1.5*\len, 0)$);
\node [label=above:$f$] (f) at (7.0*\len, \lc) {};
\draw [very thick,{|-|}] ($(f) + (-1.5*\len, 0)$) -- ($(f) + (1.5*\len, 0)$);
\end{tikzpicture}}
  \caption{
    \subref{fig:twin-house:graph} Twin-house. 
    \subref{fig:twin-house:model} An interval model of the complement of twin-house. 
    Twin-house is a co-interval graph not $12$-representable by $132$-avoiding words. }
  \label{fig:twin-house}
\end{figure}
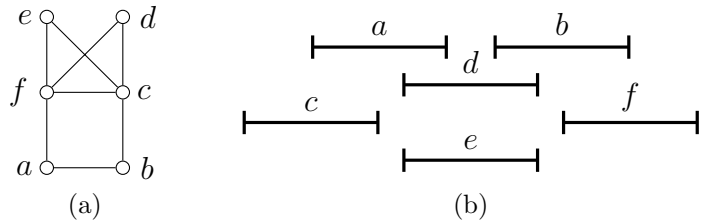

\section{$211$-representable graphs}\label{sec:211}
In this section, we show a forbidden pattern characterization 
and the related results for $211$-representable graphs. 
Apart from the forbidden patterns, 
we have a characterization of $211$-representable graphs in terms of their $12$-representants. 
\begin{theorem}\label{thm:211 word}
A labeled graph $G$ is $12$-representable by a $211$-avoiding word 
if and only if 
there is a $12$-representant $w$ of $G$ 
that can be split into two parts $w = s\pi$ such that 
$\pi$ is a permutation of all labels of $G$ and 
$s$ is the ascending order of labels occurring twice in $w$. 
\end{theorem}
\begin{proof}
The sufficiency is obvious, and we will prove the necessity. 
Let $w'$ be a $211$-avoiding representant of $G$. 
Let $i_1 < i_2 < \cdots < i_k$ be the letters occurring twice in $w'$. 
Let $w' = w_1i_1w_2i_1w_3$, where $w_1$, $w_2$, and $w_3$ are subwords of $w'$. 
No letters in $w_1$ are larger than $i_1$; 
otherwise, $w'$ would contain the pattern $211$. 
Thus, the word $w'' = i_1w_1w_2i_1w_3$ is a $12$-representant of $G$. 
The word $w''$ is still $211$-avoiding because $i_1$ is the smallest 
among the letters occurring twice in $w'$. 
Similarly, let $w'' = i_1w_1i_2w_2i_2w_3$, 
where $w_1$, $w_2$, and $w_3$ are subwords of $w''$. 
Then, the word $w''' = i_1i_2w_1w_2i_2w_3$ is again a $12$-representant of $G$ and $211$-avoiding. 
Performing the same procedure for each $i_3, \ldots, i_k$ yields a required word $w$. 
\end{proof}

Now, we show the forbidden patterns for $211$-representable graphs. 
\begin{lemma}\label{lem:211_1}
If a labeled graph contains a pattern in Figure~\ref{fig:fp211}, 
then it is not $12$-representable by a $211$-avoiding word. 
\end{lemma}

\begin{figure}[ht]
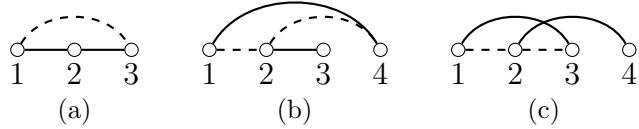

  \centering
  \subcaptionbox{\label{fig:fp211-1}}{\begin{tikzpicture}
\input{./fig/pattern/template_fp3}
\draw [thick] (x) -- (y);
\draw [thick] (y) -- (z);
\draw [thick,dashed] (x) to [out=60, in=120] (z);
\end{tikzpicture}}
  \subcaptionbox{\label{fig:fp211-2}}{\begin{tikzpicture}
\input{./fig/pattern/template_fp4}
\draw [thick,dashed] (x) to (y);
\draw [thick] (x) to [out=60, in=120] (w);
\draw [thick,dashed] (y) to [out=60, in=120] (w);
\draw [thick] (y) to (z);
\end{tikzpicture}}
  \subcaptionbox{\label{fig:fp211-3}}{\begin{tikzpicture}
\input{./fig/pattern/template_fp4}
\draw [thick,dashed] (x) to (y);
\draw [thick] (x) to [out=60, in=120] (z);
\draw [thick,dashed] (y) to (z);
\draw [thick] (y) to [out=60, in=120] (w);
\end{tikzpicture}}
  \caption{Forbidden patterns of $211$-representable graphs. }
  \label{fig:fp211}
\end{figure}

Before proving Lemma~\ref{lem:211_1}, we define a notion. 
\begin{definition}
Let $G$ be a labeled graph. 
An \emph{umbrella} is a triple of vertices $a < b < c$ 
such that $ab, bc \notin E(G)$ and $ac \in E(G)$. 
We refer to a vertex $v$ of $G$ as a \emph{$b$-vertex} if 
there exist two vertices $a$ and $c$ such that 
$a, v, c$ form an umbrella. 
\end{definition}

It is straightforward to see that the following proposition holds. 
\begin{proposition}\label{prop:b-vertices}
A labeled graph $G$ contains no pattern 
in Figures~\ref{fig:fp211}\subref{fig:fp211-2} and ~\ref{fig:fp211}\subref{fig:fp211-3} 
if and only if 
for any $b$-vertex $v$ of $G$, 
there is no vertex $d$ such that $v < d$ and $vd \in E(G)$. 
\end{proposition}

Now, we are ready to prove Lemma~\ref{lem:211_1}. 
\begin{proof}[Proof of Lemma~\ref{lem:211_1}]
By Proposition~\ref{prop:labeled induced subgraph}, 
it suffices to show that any pattern in Figure~\ref{fig:fp211} 
is not $12$-representable by a $211$-avoiding word. 
Theorem~\ref{thm:fp} implies that 
the pattern in Figure~\ref{fig:fp211}\subref{fig:fp211-1} is not $12$-representable. 
\par
We use Proposition~\ref{prop:b-vertices}. 
Let $H$ be a graph consisting of four vertices $a, b, c, d$ 
such that $a < b < c$ and $b < d$ 
together with $ab, bc \notin E(H)$ and $ac, bd \in E(H)$. 
Suppose that there is a $12$-representant $w$ of $H$. 
Since $ac \in E(H)$, 
the rightmost $c$ occurs before the leftmost $a$ in $w$. 
Since $ab \notin E(H)$, some $b$ occurs after the leftmost $a$. 
Since $bc \notin E(H)$, some $b$ occurs before the rightmost $c$. 
Thus, the pattern $bcab$ must occur in $w$. 
Since $bd \in E(H)$, every $d$ occurs before the leftmost $b$. 
Thus, $d$ and $b$ form the pattern $211$. 
Therefore, the patterns 
in Figures~\ref{fig:fp211}\subref{fig:fp211-2} and ~\ref{fig:fp211}\subref{fig:fp211-3} 
are not $12$-representable by $123$-avoiding words. 
\end{proof}

\begin{lemma}\label{lem:211_2}
If a labeled graph does not contain any pattern in Figure~\ref{fig:fp211}, 
then it is $12$-representable by a $211$-avoiding word. 
\end{lemma}
\begin{proof}
Let $G$ be a labeled graph containing no pattern in Figure~\ref{fig:fp211}. 
We first define a series of graphs $G_1 = G, G_2, \ldots, G_n$ recursively as follows. 
For each $i$ with $2 \leq i \leq n$, 
let $G_i$ be the graph obtained from $G_{i-1}$ by adding an edge $ij$ 
if there is an umbrella $k, i, j$ in $G_{i-1}$. 
The following claims imply that 
$G_n$ contains neither an umbrella nor 
the pattern in Figure~\ref{fig:fp211}\subref{fig:fp211-1}. 
\begin{claim}\label{claim:211_2_1}
No vertex of $G_{i-1}$ becomes a $b$-vertex in $G_i$. 
\end{claim}
\begin{proof}[Proof of Claim]
Suppose that a vertex $j$ is a $b$-vertex in $G_i$ but not in $G_{i-1}$. 
Then, there is a vertex $k$ with $i < j < k$ such that 
$ij, jk \notin E(G_i)$ and $ik \in E(G_i)$ but $ik \notin E(G_{i-1})$. 
Since edge $ik$ is added, there is a vertex $\ell < i$ 
such that ${\ell}i \notin E(G_{i-1})$ and ${\ell}k \in E(G_{i-1})$. 
Since $j$ is not a $b$-vertex in $G_{i-1}$ and $jk \notin E(G_{i-1})$, 
we have ${\ell}j \in E(G_{i-1})$. 
Now, the vertices ${\ell}, i, j$ form an umbrella in $G_{i-1}$ 
but $ij \notin E(G_i)$, a contradiction. 
\end{proof}
Therefore, every vertex less than $i$ is not a $b$-vertex in $G_{i-1}$. 
\begin{claim}\label{claim:211_2_2}
If $G_{i-1}$ contains no pattern in Figure~\ref{fig:fp211}\subref{fig:fp211-1}, then $G_i$ does. 
\end{claim}
\begin{proof}[Proof of Claim]
Suppose first that $G_i$ contains the pattern in Figure~\ref{fig:fp211}\subref{fig:fp211-1} 
consisting of vertices $i < j < k$. 
Since $G_{i-1}$ does not contain the pattern, 
$ik \notin E(G_i)$ and $ij, jk \in E(G_i)$ but $ij \notin E(G_{i-1})$. 
Since edge $ij$ is added, there is a vertex $\ell < i$ 
such that ${\ell}i \notin E(G_{i-1})$ and ${\ell}j \in E(G_{i-1})$. 
If ${\ell}k \in E(G_{i-1})$ then 
${\ell}, i, k$ form an umbrella in $G_{i-1}$ but $ik \notin E(G_i)$, 
a contradiction. 
If ${\ell}k \notin E(G_{i-1})$ then 
${\ell}, j, k$ form the pattern in $G_{i-1}$, a contradiction. 
\par
Suppose now that $G_i$ contains the pattern in Figure~\ref{fig:fp211}\subref{fig:fp211-1} 
consisting of vertices $j < i < k$. 
Since $G_{i-1}$ does not contain the pattern, 
$jk \notin E(G_i)$ and $ji, ik \in E(G_i)$ but $ik \notin E(G_{i-1})$. 
Since edge $ik$ is added, there is a vertex $\ell < i$ 
such that ${\ell}i \notin E(G_{i-1})$ and ${\ell}k \in E(G_{i-1})$. 
Suppose first ${\ell} < j$. 
If ${\ell}j \in E(G_{i-1})$ then 
${\ell}, j, i$ form the pattern in $G_{i-1}$, a contradiction. 
If ${\ell}j \notin E(G_{i-1})$ then 
${\ell}, j, k$ form an umbrella in $G_{i-1}$, 
which contradicts the above claim. 
Suppose finally $j < {\ell}$. 
If $j{\ell} \in E(G_{i-1})$ then 
$j, {\ell}, k$ form the pattern in $G_{i-1}$, a contradiction. 
If $j{\ell} \notin E(G_{i-1})$ then 
$j, {\ell}, i$ form an umbrella in $G_{i-1}$, 
which contradicts the above claim. 
\end{proof}
\par
Now, Theorem~\ref{thm:permutation_labeled} indicates that 
there is a permutation $\pi$ representing $G_n$. 
Let $s$ be the ascending sequence of labels of the $b$-vertices of $G$. 
Then, the word $w = s\pi$ is a $12$-representant of $G$ since 
Proposition~\ref{prop:b-vertices} indicates that for any $b$-vertex $i$, 
there is no vertex $j$ such that $i < j$ and $ij \in E(G)$. 
It is obvious that $w$ is $211$-avoiding. 
\end{proof}

Lemmas~\ref{lem:211_1} and~\ref{lem:211_2} yield the main theorem of this section. 
\begin{theorem}\label{thm:211}
A labeled graph is $12$-representable by a $211$-avoiding word if and only if 
it does not contain any pattern in Figure~\ref{fig:fp211}. 
\end{theorem}

Recall that an \emph{$L$-shape} is a union of a vertical and a horizontal line segment in which 
the bottom endpoint of the vertical line segment coincides with the left endpoint of the horizontal one. 
A \emph{grounded $L$-graph}~\cite{JT19-EJC} is the intersection graph of $L$-shapes 
whose top endpoints lie on a horizontal line. 
The forbidden pattern characterization is known for grounded $L$-graphs. 
\begin{theorem}[\cite{JT19-EJC}]\label{thm:fp_grounded_L}
A graph is a grounded $L$-graph if and only if 
it has a vertex ordering 
that does not contain any pattern in Figure~\ref{fig:fp_grounded_L}. 
\end{theorem}

\begin{figure}[ht]
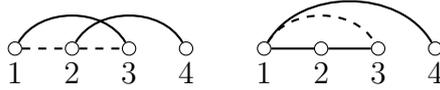

  \centering
  \begin{tikzpicture}
\input{./fig/pattern/template_fp4}
\draw [thick] (x) to [out=60, in=120] (z);
\draw [thick] (y) to [out=60, in=120] (w);
\draw [thick,dashed] (x) to (y);
\draw [thick,dashed] (y) to (z);
\end{tikzpicture}
  \begin{tikzpicture}
\input{./fig/pattern/template_fp4}
\draw [thick] (x) to [out=60, in=120] (w);
\draw [thick,dashed] (x) to [out=60, in=120] (z);
\draw [thick] (x) to (y);
\draw [thick] (y) to (z);
\end{tikzpicture}
  \caption{Forbidden patterns of grounded $L$-graphs. }
  \label{fig:fp_grounded_L}
\end{figure}

Theorems~\ref{thm:211} and~\ref{thm:fp_grounded_L} indicate that 
the class of $211$-representable graphs is a subclass of grounded $L$-graphs. 
\begin{corollary}\label{prop:211 permutation graph}
A graph is a grounded $L$-graph if it is $12$-representable by a $211$-avoiding word. 
\end{corollary}

\section{Concluding remarks}\label{sec:conclusion}
This paper investigates $p$-representable graphs 
(i.e., graphs $12$-representable by $p$-avoiding words) 
when the pattern $p$ is of length 3. 
We present forbidden pattern characterizations for most cases. 
Characterizing $112$-representable graphs 
(equivalently, $122$-representable graphs) 
remains an open problem. 

A natural next step is to study $p$-representable graphs for longer patterns. 
The following observation indicates that 
for any $k \geq 2$, 
there is a pattern $p$ of length $k$ such that 
the class of $p$-representable graphs is nontrivial. 
\begin{proposition}
No clique of size greater than or equal to $k \geq 2$ is not $p$-representable 
if $p = k (k-1) \cdots 1$. 
\end{proposition}

More generally, it is a further research direction 
to study graphs $12$-representable by 
some specific types of words, 
as suggested in~\cite[page~183]{KL15-book}. 
One example is graphs defined by words that avoid a set of patterns. 
Let $\Gamma$ be a set of patterns such that $\red(p) = p$ for all $p \in \Gamma$. 
We say that a graph is \emph{$\Gamma$-representable} if 
it is $12$-representable by a word that avoids all patterns in $\Gamma$. 
For $\{121, 212\}$-representable graphs, the following can be obtained from Theorems~\ref{thm:121} and~\ref{thm:permutation}. 
\begin{theorem}
A graph is $12$-representable by a word that avoids all patterns in $\{121, 212\}$ 
if and only if it is a permutation graph. 
\end{theorem}
Furthermore, the following holds for $\{211, 221\}$-representable graphs. 
\begin{theorem}
A graph is $12$-representable by a word that avoids all patterns in $\{211, 221\}$ 
if and only if it is a permutation graph. 
\end{theorem}
\begin{proof}
The sufficiency is obvious from Theorem~\ref{thm:permutation}, and we will prove the necessity. 
Let $G$ be a graph and 
$w$ be a $12$-representant of $G$ that avoids the patterns $211$ and $221$. 
Suppose that a letter $i$ occurs twice in $w$. 
Let $w = w_1 i w_2 i w_3$, where $w_1$, $w_2$, and $w_3$ are subwords of $w$. 
Since $w$ is $211$-avoiding, no letters in $w_1$ are larger than $i$. 
Recall that $w$ contains at least one copy of each letter in $[n]$. 
Thus, all letters larger than $i$ occur after the first copy of $i$. 
Similarly, 
since $w$ is $221$-avoiding, no letters in $w_3$ are less than $i$. 
Thus, all letters less than $i$ occur before the second copy of $i$. 
Therefore, the vertex labeled with $i$ is isolated in $G$. 
We relabel $G$ so that label $i$ is replaced by $n+1$. 
The word $w' = w_1w_2w_3(n+1)$ is a $12$-representant of the relabeled graph 
and still avoids the patterns $211$ and $221$. 
Performing the same procedure with all other letters occurring twice in $w'$ 
yields a permutation representing $G$. 
\end{proof}

Another research direction is the recognition. 
The complexity of recognition problems has been known for the classes 
we dealt with in Section~\ref{sec:simple cases}: 
simple-triangle graphs can be recognized 
in polynomial time~\cite{Mertzios15-SIAMDM,Takaoka20a-DAM,Takaoka20-DAM}; 
permutation graphs, 
trivially perfect graphs, and 
bipartite permutation graphs 
can be recognized in linear time~\cite{FH21-SIDMA}. 
The cases of 
$123$-representable graphs, 
$132$-representable graphs, and 
$211$-representable graphs 
are left as open problems.

\section*{Acknowledgments}
This work was supported by JSPS KAKENHI Grant Number JP23K03191. 

\bibliographystyle{abbrvurl}
\bibliography{ref}

\begin{thebibliography}{10}

\bibitem{BW99-DM}
K.~P. Bogart and D.~B. West.
\newblock A short proof that 'proper = unit'.
\newblock {\em Discrete Math.}, 201(1-3):21--23, 1999.
\newblock \href {http://dx.doi.org/10.1016/S0012-365X(98)00310-0}
  {\path{doi:10.1016/S0012-365X(98)00310-0}}.

\bibitem{BB23-DAM}
F.~Bonomo{-}Braberman and G.~A. Brito.
\newblock Intersection models and forbidden pattern characterizations for
  2-thin and proper 2-thin graphs.
\newblock {\em Discret. Appl. Math.}, 339:53--77, 2023.
\newblock \href {http://dx.doi.org/10.1016/j.dam.2023.06.013}
  {\path{doi:10.1016/j.dam.2023.06.013}}.

\bibitem{BLS99}
A.~Brandst\"{a}dt, V.~B. Le, and J.~P. Spinrad.
\newblock {\em Graph Classes: A Survey}.
\newblock {SIAM}, Philadelphia, PA, USA, 1999.
\newblock \href {http://dx.doi.org/10.1137/1.9780898719796}
  {\path{doi:10.1137/1.9780898719796}}.

\bibitem{CCFHHHS17-DAM}
D.~Catanzaro, S.~Chaplick, S.~Felsner, B.~V. Halld{\'{o}}rsson, M.~M.
  Halld{\'{o}}rsson, T.~Hixon, and J.~Stacho.
\newblock Max point-tolerance graphs.
\newblock {\em Discret. Appl. Math.}, 216:84--97, 2017.
\newblock \href {http://dx.doi.org/10.1016/j.dam.2015.08.019}
  {\path{doi:10.1016/j.dam.2015.08.019}}.

\bibitem{CK22-DMGT}
J.~Chen and S.~Kitaev.
\newblock On the 12-representability of induced subgraphs of a grid graph.
\newblock {\em Discuss. Math. Graph Theory}, 42(2):383--403, 2022.
\newblock \href {http://dx.doi.org/10.7151/dmgt.2263}
  {\path{doi:10.7151/dmgt.2263}}.

\bibitem{CK87-CN}
D.~G. Corneil and P.~A. Kamula.
\newblock Extensions of permutation and interval graphs.
\newblock {\em Congr. Numer.}, 58:267--275, 1987.

\bibitem{Damaschke90-incollection}
P.~Damaschke.
\newblock Forbidden ordered subgraphs.
\newblock In R.~Bodendiek and R.~Henn, editors, {\em Topics in Combinatorics
  and Graph Theory: Essays in Honour of Gerhard Ringel}, pages 219--229.
  Physica-Verlag HD, Heidelberg, 1990.
\newblock \href {http://dx.doi.org/10.1007/978-3-642-46908-4_25}
  {\path{doi:10.1007/978-3-642-46908-4_25}}.

\bibitem{FHH99-Combinatorica}
T.~Feder, P.~Hell, and J.~Huang.
\newblock List homomorphisms and circular arc graphs.
\newblock {\em Combinatorica}, 19(4):487--505, 1999.
\newblock \href {http://dx.doi.org/10.1007/s004939970003}
  {\path{doi:10.1007/s004939970003}}.

\bibitem{FHHR12-DAM}
T.~Feder, P.~Hell, J.~Huang, and A.~Rafiey.
\newblock Interval graphs, adjusted interval digraphs, and reflexive list
  homomorphisms.
\newblock {\em Discrete Appl. Math.}, 160(6):697--707, 2012.
\newblock \href {http://dx.doi.org/10.1016/j.dam.2011.04.016}
  {\path{doi:10.1016/j.dam.2011.04.016}}.

\bibitem{FH21-SIDMA}
L.~Feuilloley and M.~Habib.
\newblock Graph classes and forbidden patterns on three vertices.
\newblock {\em {SIAM} J. Discret. Math.}, 35(1):55--90, 2021.
\newblock \href {http://dx.doi.org/10.1137/19M1280399}
  {\path{doi:10.1137/19M1280399}}.

\bibitem{GKZ17-AJC}
A.~L.~L. Gao, S.~Kitaev, and P.~B. Zhang.
\newblock On 132-representable graphs.
\newblock {\em Australas. J. Combin.}, 69:105--118, 2017.
\newblock URL: \url{http://ajc.maths.uq.edu.au/pdf/69/ajc\_v69\_p105.pdf}.

\bibitem{Golumbic04}
M.~C. Golumbic.
\newblock {\em Algorithmic Graph Theory and Perfect Graphs}, volume~57 of {\em
  Ann. Discrete Math.}
\newblock Elsevier, 2 edition, 2004.

\bibitem{GHH23-AMC}
S.~Guzm{\'{a}}n{-}Pro, P.~Hell, and C.~Hern{\'{a}}ndez{-}Cruz.
\newblock Describing hereditary properties by forbidden circular orderings.
\newblock {\em Appl. Math. Comput.}, 438:127555, 2023.
\newblock \href {http://dx.doi.org/10.1016/j.amc.2022.127555}
  {\path{doi:10.1016/j.amc.2022.127555}}.

\bibitem{HM20-ALGO}
M.~Habib and L.~Mouatadid.
\newblock Maximum induced matching algorithms via vertex ordering
  characterizations.
\newblock {\em Algorithmica}, 82(2):260--278, 2020.
\newblock \href {http://dx.doi.org/10.1007/s00453-018-00538-5}
  {\path{doi:10.1007/s00453-018-00538-5}}.

\bibitem{HKP16-DAM}
M.~M. Halld{\'{o}}rsson, S.~Kitaev, and A.~V. Pyatkin.
\newblock Semi-transitive orientations and word-representable graphs.
\newblock {\em Discret. Appl. Math.}, 201:164--171, 2016.
\newblock \href {http://dx.doi.org/10.1016/j.dam.2015.07.033}
  {\path{doi:10.1016/j.dam.2015.07.033}}.

\bibitem{HHMR20-SIDMA}
P.~Hell, J.~Huang, R.~M. McConnell, and A.~Rafiey.
\newblock Min-orderable digraphs.
\newblock {\em {SIAM} J. Discret. Math.}, 34(3):1710--1724, 2020.
\newblock \href {http://dx.doi.org/10.1137/19M1241763}
  {\path{doi:10.1137/19M1241763}}.

\bibitem{HMR14-LNCS}
P.~Hell, B.~Mohar, and A.~Rafiey.
\newblock Ordering without forbidden patterns.
\newblock In A.~Schulz and D.~Wagner, editors, {\em {ESA} 2014}, volume 8737 of
  {\em Lecture Notes in Comput. Sci.}, pages 554--565. Springer Berlin
  Heidelberg, 2014.
\newblock \href {http://dx.doi.org/10.1007/978-3-662-44777-2_46}
  {\path{doi:10.1007/978-3-662-44777-2_46}}.

\bibitem{Hixon13-Master}
T.~S. Hixon.
\newblock Hook graphs and more: Some contributions to geometric graph theory.
\newblock Master's thesis, Technische Universit\"{a}t Berlin, 2013.

\bibitem{Huang18-DAM}
J.~Huang.
\newblock Non-edge orientation and vertex ordering characterizations of some
  classes of bigraphs.
\newblock {\em Discret. Appl. Math.}, 245:190--193, 2018.
\newblock \href {http://dx.doi.org/10.1016/j.dam.2017.02.001}
  {\path{doi:10.1016/j.dam.2017.02.001}}.

\bibitem{JT19-EJC}
V.~Jel{\'{\i}}nek and M.~T{\"{o}}pfer.
\newblock On grounded {L}-graphs and their relatives.
\newblock {\em Electron. J. Comb.}, 26(3):3, 2019.
\newblock \href {http://dx.doi.org/10.37236/8096} {\path{doi:10.37236/8096}}.

\bibitem{JKPR15-EJC}
M.~E. Jones, S.~Kitaev, A.~V. Pyatkin, and J.~B. Remmel.
\newblock Representing graphs via pattern avoiding words.
\newblock {\em Electron. J. Comb.}, 22(2):P2.53, 2015.
\newblock \href {http://dx.doi.org/10.37236/4946} {\path{doi:10.37236/4946}}.

\bibitem{Kitaev17-LNCS}
S.~Kitaev.
\newblock A comprehensive introduction to the theory of word-representable
  graphs.
\newblock In {\'{E}}.~Charlier, J.~Leroy, and M.~Rigo, editors, {\em {DLT}
  2017}, volume 10396 of {\em Lecture Notes in Comput. Sci.}, pages 36--67.
  Springer, 2017.
\newblock \href {http://dx.doi.org/10.1007/978-3-319-62809-7\_2}
  {\path{doi:10.1007/978-3-319-62809-7\_2}}.

\bibitem{Kitaev17-JGT}
S.~Kitaev.
\newblock Existence of \emph{u}-representation of graphs.
\newblock {\em J. Graph Theory}, 85(3):661--668, 2017.
\newblock \href {http://dx.doi.org/10.1002/jgt.22097}
  {\path{doi:10.1002/jgt.22097}}.

\bibitem{KL15-book}
S.~Kitaev and V.~V. Lozin.
\newblock {\em Words and Graphs}.
\newblock Monographs in Theoretical Computer Science. An {EATCS} Series.
  Springer, 2015.
\newblock \href {http://dx.doi.org/10.1007/978-3-319-25859-1}
  {\path{doi:10.1007/978-3-319-25859-1}}.

\bibitem{KP08-JALC}
S.~Kitaev and A.~V. Pyatkin.
\newblock On representable graphs.
\newblock {\em J. Autom. Lang. Comb.}, 13(1):45--54, 2008.
\newblock \href {http://dx.doi.org/10.25596/jalc-2008-045}
  {\path{doi:10.25596/jalc-2008-045}}.

\bibitem{LHKLM19-TCS}
F.~D. Luca, M.~I. Hossain, S.~G. Kobourov, A.~Lubiw, and D.~Mondal.
\newblock Recognition and drawing of stick graphs.
\newblock {\em Theor. Comput. Sci.}, 796:22--33, 2019.
\newblock \href {http://dx.doi.org/10.1016/j.tcs.2019.08.018}
  {\path{doi:10.1016/j.tcs.2019.08.018}}.

\bibitem{Mandelshtam19-DMGT}
Y.~Mandelshtam.
\newblock On graphs presentable by pattern-avoiding words.
\newblock {\em Discuss. Math. Graph Theory}, 39(2):375--389, 2019.
\newblock \href {http://dx.doi.org/10.7151/dmgt.2128}
  {\path{doi:10.7151/dmgt.2128}}.

\bibitem{Mertzios15-SIAMDM}
G.~B. Mertzios.
\newblock The recognition of simple-triangle graphs and of linear-interval
  orders is polynomial.
\newblock {\em {SIAM} J. Discrete Math.}, 29(3):1150--1185, 2015.
\newblock \href {http://dx.doi.org/10.1137/140963108}
  {\path{doi:10.1137/140963108}}.

\bibitem{MNN18-SIDMA}
G.~B. Mertzios, A.~Nichterlein, and R.~Niedermeier.
\newblock A linear-time algorithm for maximum-cardinality matching on
  cocomparability graphs.
\newblock {\em {SIAM} J. Discret. Math.}, 32(4):2820--2835, 2018.
\newblock \href {http://dx.doi.org/10.1137/17M1120920}
  {\path{doi:10.1137/17M1120920}}.

\bibitem{Olariu91-IPL}
S.~Olariu.
\newblock An optimal greedy heuristic to color interval graphs.
\newblock {\em Inf. Process. Lett.}, 37(1):21--25, 1991.
\newblock \href {http://dx.doi.org/10.1016/0020-0190(91)90245-D}
  {\path{doi:10.1016/0020-0190(91)90245-D}}.

\bibitem{Rafiey22-JGT}
A.~Rafiey.
\newblock Recognizing interval bigraphs by forbidden patterns.
\newblock {\em J. Graph Theory}, 100(3):504--529, 2022.
\newblock \href {http://dx.doi.org/10.1002/jgt.22792}
  {\path{doi:10.1002/jgt.22792}}.

\bibitem{Rotem81-DM}
D.~Rotem.
\newblock Stack sortable permutations.
\newblock {\em Discret. Math.}, 33(2):185--196, 1981.
\newblock \href {http://dx.doi.org/10.1016/0012-365X(81)90165-5}
  {\path{doi:10.1016/0012-365X(81)90165-5}}.

\bibitem{Rusu23-TCS}
I.~Rusu.
\newblock On the complexity of recognizing {Stick}, {BipHook} and {Max
  Point-Tolerance} graphs.
\newblock {\em Theor. Comput. Sci.}, 952:113773, 2023.
\newblock \href {http://dx.doi.org/10.1016/j.tcs.2023.113773}
  {\path{doi:10.1016/j.tcs.2023.113773}}.

\bibitem{STU10-DAM}
A.~M.~S. Shrestha, S.~Tayu, and S.~Ueno.
\newblock On orthogonal ray graphs.
\newblock {\em Discrete Appl. Math.}, 158(15):1650--1659, 2010.
\newblock \href {http://dx.doi.org/10.1016/j.dam.2010.06.002}
  {\path{doi:10.1016/j.dam.2010.06.002}}.

\bibitem{SC15-DMTCS}
M.~Soto and C.~T. Caro.
\newblock {$p$-Box}: {A} new graph model.
\newblock {\em Discret. Math. Theor. Comput. Sci.}, 17(1):169--186, 2015.
\newblock \href {http://dx.doi.org/10.46298/dmtcs.2121}
  {\path{doi:10.46298/dmtcs.2121}}.

\bibitem{Spinrad03}
J.~P. Spinrad.
\newblock {\em Efficient Graph Representations}, volume~19 of {\em Fields
  Institute Monographs}.
\newblock {AMS}, Providence, RI, USA, 2003.

\bibitem{Takaoka18-DM}
A.~Takaoka.
\newblock A vertex ordering characterization of simple-triangle graphs.
\newblock {\em Discrete Math.}, 341(12):3281--3287, 2018.
\newblock \href {http://dx.doi.org/10.1016/j.disc.2018.08.009}
  {\path{doi:10.1016/j.disc.2018.08.009}}.

\bibitem{Takaoka20a-DAM}
A.~Takaoka.
\newblock A recognition algorithm for simple-triangle graphs.
\newblock {\em Discret. Appl. Math.}, 282:196--207, 2020.
\newblock \href {http://dx.doi.org/10.1016/j.dam.2019.11.009}
  {\path{doi:10.1016/j.dam.2019.11.009}}.

\bibitem{Takaoka20-DAM}
A.~Takaoka.
\newblock Recognizing simple-triangle graphs by restricted 2-chain subgraph
  cover.
\newblock {\em Discret. Appl. Math.}, 279:154--167, 2020.
\newblock \href {http://dx.doi.org/10.1016/j.dam.2019.10.028}
  {\path{doi:10.1016/j.dam.2019.10.028}}.

\bibitem{Takaoka21-DAM}
A.~Takaoka.
\newblock A recognition algorithm for adjusted interval digraphs.
\newblock {\em Discret. Appl. Math.}, 294:253--256, 2021.
\newblock \href {http://dx.doi.org/10.1016/j.dam.2021.02.014}
  {\path{doi:10.1016/j.dam.2021.02.014}}.

\bibitem{Takaoka23a-arXiv}
A.~Takaoka.
\newblock Computing shortest 12-representants of labeled graphs.
\newblock {\em CoRR}, abs/2304.07507, 2023.
\newblock \href {http://arxiv.org/abs/2304.07507} {\path{arXiv:2304.07507}},
  \href {http://dx.doi.org/10.48550/arXiv.2304.07507}
  {\path{doi:10.48550/arXiv.2304.07507}}.

\bibitem{Takaoka23-DMGT-inpress}
A.~Takaoka.
\newblock Graph classes equivalent to 12-representable graphs.
\newblock {\em Discuss. Math. Graph Theory}, in press.
\newblock \href {http://dx.doi.org/10.7151/dmgt.2486}
  {\path{doi:10.7151/dmgt.2486}}.

\bibitem{TM76-DM}
W.~T. {Trotter, Jr.} and J.~I. {Moore, Jr.}
\newblock Characterization problems for graphs, partially ordered sets,
  lattices, and families of sets.
\newblock {\em Discrete Math.}, 16(4):361--381, 1976.
\newblock \href {http://dx.doi.org/10.1016/S0012-365X(76)80011-8}
  {\path{doi:10.1016/S0012-365X(76)80011-8}}.

\end{thebibliography}

\end{document}